\documentclass{amsart}
\usepackage{a4wide}

\usepackage[all]{xy}

\usepackage{amsfonts}
\usepackage{amsmath}
\usepackage{amsthm}
\usepackage{bbm}
\usepackage{eufrak}
\usepackage{mathrsfs}
\usepackage{euscript}
\usepackage{mathabx}

\DeclareMathAlphabet{\mathpzc}{OT1}{pzc}{m}{it}

\newcommand{\ep}[1]{\varepsilon_{#1}}

\newcommand{\End}{\mathrm{End}}

\newcommand{\car}{\mathrm{char}\,}
\newcommand{\ann}{\mathrm{ann}}
\newcommand{\kk}{\mathbbm{k}}

\newcommand{\FF}{\mathbbm{F}}

\newcommand{\psmash}[2]{\underline{{#1} \# {#2}}}
\newcommand{\modpar}[2]{({#1},{#2})}
\newcommand{\ppsmash}[2]{\left( \underline{{#1} \# {#2}} \right)}

\newcommand{\dual}[1]{{#1}^\varstar}

\newcommand{\br}[1]{\bar{#1}}
\newcommand{\pH}{\underline{H}}
\newcommand{\pcoH}{\underline{coH}}
\newcommand{\rz}[1]{\sqrt{#1}}
\newcommand{\Hrz}[1]{\sqrt[H]{#1}}
\newcommand{\comp}[2]{{#1} \backslash {#2}}
\newcommand{\mpzc}[1]{\mathpzc{#1}}
\newcommand{\mc}[1]{\mathcal{#1}}
\newcommand{\mf}[1]{\mathfrak{#1}}

\newcommand{\quoc}[2]{{#1} / {#2}}

\newcommand{\fun}[3]{{#1} : {#2} \rightarrow {#3}}

\newcommand{\ffun}[4]{{#1} & \rightarrow & {#2} \\ {#3} & \mapsto & {#4}}
\newcommand{\ffuni}[4]{{#1} & \hookrightarrow & {#2} \\ {#3} & \mapsto & {#4}}

\newcommand{\funn}[5]{{#1} : {#2} & \rightarrow & {#3} \\ {#4} & \mapsto & {#5}}

\newcommand{\conj}[2]{\left\{ {#1} \; : \; {#2} \right\}}
\newcommand{\paren}[1]{\left( {#1} \right)}

\newcommand{\ideal}{\unlhd}

\newcommand{\ts}{\otimes}

\newcommand{\tssp}{\, {\scriptstyle \otimes} \,}
\newcommand{\tsspp}{\, {\scriptscriptstyle \otimes} \,}
\newcommand{\ace}{\blacktriangleright}
\newcommand{\acd}{\blacktriangleleft}
\newcommand{\acind}{\smalltriangleright}
\newcommand{\itm}[1]{\textsf{#1}}

\newcommand{\esp}{\vspace{.2cm}}

\newtheorem{teo}{Theorem}[section]
\newtheorem{lema}[teo]{Lemma}
\newtheorem{cor}[teo]{Corollary}
\newtheorem{prop}[teo]{Proposition}

\newtheorem{defi}[teo]{Definition}
\newtheorem{ex}[teo]{Example}

\newtheorem{remark}[teo]{Remark}

\theoremstyle{definition}

\newtheorem{qt}{Question}

\title[Semiprimitivity and semiprimality of partial smash products]{On the semiprimitivity and the semiprimality problems for  partial smash products}

\author[R. Cavalheiro]{Rafael Cavalheiro}
\address{Instituto de Matem\'{a}tica, Universidade Federal do Rio Grande do Sul, Brazil}
\email{rafael.cavalheiro@ufrgs.br}
\author[A. Sant'Ana]{Alveri Sant'Ana}
\address{Instituto de Matem\'{a}tica, Universidade Federal do Rio Grande do Sul, Brazil}
\email{alveri@mat.ufrgs.br}

\thanks{ The first author was supported by CNPq and CAPES - Brazil. Most  of the results presented here were obtained during the doctoral studies of the first author at the Federal University of Rio Grande do Sul (UFRGS), Brazil}

\subjclass[2010]{16N20, 16N60, 16R99, 16S40, 16T99}

\date{September, 2014.}

\keywords{partial Hopf actions, partial smash product, semiprimitivity, semiprimality, $H$-radicals, partial $(A,H)$-modules }

\begin{document}

\maketitle

\begin{abstract}
 In this paper we discuss about the semiprimitivity and the semiprimality of  partial smash products. Let $H$ be a semisimple Hopf algebra over a field $\kk$ and let $A$ be a left partial $H$-module algebra. We study the $H$-prime and the $H$-Jacobson radicals of $A$ and its relations with the prime and the Jacobson radicals of $\psmash{A}{H}$, respectively. In particular, we prove that if $A$ is $H$-semiprimitive, then $\psmash{A}{H}$ is semiprimitive provided that all irreducible representations of $A$ are finite-dimensional, or $A$ is an affine PI-algebra over $\kk$ and $\kk$ is a perfect field, or $A$ is locally finite. Moreover, we prove that $\psmash{A}{H}$ is semiprime provided that $A$ is an $H$-semiprime PI-algebra, generalizing for the setting of partial actions, the main results of \cite{LMS} and \cite{LM}.
\end{abstract}


\section{Introduction}

Let $H$ be a finite-dimensional Hopf algebra over a field $\kk$ and let $A$ be a left $H$-module algebra. An important question in the theory of Hopf algebra actions is to know when the smash product $A \# H$ is semiprime. It is well known that if $A$ is semiprime, then $A \# H$ is semiprime in the following cases: if $H = \kk G$ and $|G|^{-1} \in \kk$ \cite{FM}; or if $H = \dual{(\kk G)}$ \cite{CM}. In both cases $H$ is semisimple Artinian. This suggests the following question (raised by Cohen and Fishman in \cite{CF}):

\begin{qt} \label{qsemiprimal1}
\emph{Let $H$ be a semisimple Hopf algebra and let $A$ be a semiprime $H$-module algebra. Is the smash product $A \# H$ also semiprime?}
\end{qt}

Many special cases of the Question \ref{qsemiprimal1} have been answered by adding hypotheses on $H$ or on $A$ (e.g. \cite{MS}, \cite{Lo1}, \cite{LMS} and \cite{LM}). In \cite{LMS} the authors tackled this question by studying the stability of the Jacobson radical. From this approach, naturally new related questions arisen about the semiprimitivity and the semiprimality of the smash product:

\begin{qt} \label{qsemiprimit}
\emph{Let $H$ be a  semisimple Hopf algebra over a field $\kk$ and $A$ an $H$-semiprimitive (resp. $H$-semiprime) $H$-module algebra. Is the smash product $A \# H$  semiprimitive (resp. semiprime)?}
\end{qt}

In \cite{LMS} the authors showed that if $\kk$ has characteristic $0$, then the Question \ref{qsemiprimit}  has an affirmative answer for the semiprimitivity case provided that $A$ is a PI-algebra which is either affine or algebraic over $\kk$, or all irreducible representations of $A$ are finite-dimensional, or $A$ is locally finite; if $\kk$ has positive characteristic then additional hypotheses were assumed. In \cite{LM} the authors showed that the answer for the case of semiprimality in Question \ref{qsemiprimit} is `yes' provided that $A$ is a PI-algebra.

Partial group actions were first defined by R. Exel in the context of operator algebras in the study of $\dual{C}$-algebras generated by partial isometries on a Hilbert space \cite{E}. In \cite{DE} partial group actions were defined axiomatically and in \cite{CJ} the authors extended these concepts for the context of partial Hopf actions. Since then many papers were published in this subject and the partial actions became an independent area of research in ring theory. Significant advances were made in this area as, for example, Galois theory (\cite{DFP}, \cite{CJ}), Morita theory (\cite{AB2}, \cite{ADES}) and partial representations (\cite{DE}, \cite{AB}).

In this work we generalize the main results of \cite{LMS} and \cite{LM} for the context of partial Hopf actions. More precisely, we consider a semisimple Hopf algebra $H$ over a field $\kk$ and a left partial $H$-module algebra $A$. We prove that if $A$ is $H$-semiprimitive then $\psmash{A}{H}$ is semiprimitive provided that all irreducible representations of $A$ are finite-dimensional (Theorem \ref{Semiprimitiv1}); or $A$ is an affine PI-algebra over $\kk$ and $\kk$ is a perfect field (Theorem \ref{Semiprimitiv2}); or $A$ is locally finite (Theorem \ref{Semiprimitiv3}). Differently of \cite{LMS} we not suppose additional hypotheses when $\kk$ has positive characteristic (except in the second case before). We also prove that $\psmash{A}{H}$ is semiprime when $A$ is $H$-semiprime and satisfies a polynomial identity (Theorem \ref{semiprimal1}).

For our purposes we study the concepts of $H$-prime and $H$-Jacobson radicals of a partial $H$-module algebra, which arise naturally from the concepts of $H$-stable ideal and partial $\modpar{A}{H}$-modules as a generalization of the concepts of prime and Jacobson radicals of an any algebra. This concepts of $H$-stable ideal and partial $\modpar{A}{H}$-modules are adaptations, for the case of partial actions, of the correspondent concepts of $H$-stable ideal and $\modpar{A}{H}$-modules studied, for example, in \cite{F}, \cite{MS} and \cite{Z} for global actions.


\section{Partial (co)actions and partial smash products}

Throughout this paper $H$ will denote a Hopf algebra over a field $\kk$ (and, unless mentioned otherwise, all algebras will be over the same field $\kk$). Also we use the {\it Sweedler's notation} to denote the comultiplication of an element $h\in H$, that is, $\Delta(h) = \sum h_1 \otimes h_2$. We start recalling some well-known concepts which are fundamentals for this paper. For more details we refer \cite{CJ}.

\esp
\begin{defi} \label{defHmodAlgPar}
A (left) partial action of a Hopf algebra $H$ on an algebra $A$ is a $\kk$-linear map $\xi: H \ts A \to A$, $h \ts a \mapsto h \cdot a$, such that, for any $a,b \in A$ and $h,g \in H$, we have
\begin{itemize}
\item[\itm{(PA1)}] $1_H \cdot a = a$
\item[\itm{(PA2)}] $h \cdot (a(g \cdot b)) = \sum (h_1 \cdot a)((h_2 g) \cdot b).$
\end{itemize}
In this case, $A$ is called a (left) partial $H$-module algebra.
\end{defi}
\esp

Of course, we can define right partial action of $H$ on $A$, but in this paper we will consider only left actions of $H$ on $A$ so that the expression {\it partial $H$-module algebra} will means {\it left partial $H$-module algebra}. The same occur later for coactions, but on the right.

In this paper we will consider only unital algebras. In such a case, the condition $\itm{(PA2)}$ of the Definition \ref{defHmodAlgPar} can be replaced for the following two conditions: for any $a,b \in A$ and $h,g \in H$,
\begin{itemize}
\item[\itm{(PA3)}] $h \cdot (ab) = \sum (h_1 \cdot a)(h_2 \cdot b)$
\item[\itm{(PA4)}] $h \cdot (g \cdot b) = \sum (h_1 \cdot 1_A)((h_2 g) \cdot b).$
\end{itemize}

It is easy to verify that if $A$ is a partial $H$-module algebra, then
$$h \cdot 1_A = \varepsilon(h) 1_A \, , \; \forall \, h \in H \quad \Longleftrightarrow \quad h \cdot (g \cdot a) = (h g) \cdot a \, , \; \forall \, a \in A, \, \forall \, h,g \in H.$$
 When this is the case, $A$ is called an {\it $H$-module algebra} (see \cite[Definition 6.1.1]{DNR}) and $\xi$ is called {\it global action}. On the other hand, is immediate to see that any $H$-module algebra is a partial $H$-module algebra.


\esp
\begin{ex} \label{AcParcInduz1}   \emph{\cite[Proposition 1]{AB}}
Let $B$ be an $H$-module algebra (global action) and let $A$ be a right ideal of $B$ with identity element $1_A$. Then $A$ becomes a partial $H$-module algebra by
$$h \cdot a := 1_A(h \acind a), \quad \forall \, a \in A, \, \forall \, h \in H$$
where $\acind$ indicates the action of $H$ on $B$.
\end{ex}
\esp

As a particular case of this example, the authors in  \cite[p. 5]{AB2}   have considered a finite group $G$ and the Hopf algebra $H = \dual{(\kk G)}$. Then $B = \kk G$ is an $H$-module algebra by
$$p_g \acind h = \delta_{g,h} h \, , \; \forall \, g, h \in G.$$
Now consider a normal subgroup $N \unlhd G$, $N \neq \{ 1_G \}$, such that $\car \kk \notdivides |N|$. Thus, the element
$$e_N = \frac{1}{|N|} \sum_{n \in N} n$$
is a central idempotent in $G$ as it is easy to see.
Hence, the ideal $A = e_N B$ is a unital $\kk$-algebra with $1_A = e_N$ and the action induced of $H$ on $A$ as in Example \ref{AcParcInduz1} is such that, for every $g \in N$,
$$p_g \cdot e_N = e_N (p_g \acind e_N) = (1/|N|) e_N g = (1/|N|) e_N \neq \delta_{1,g} e_N = \varepsilon(p_g) e_N.$$
In particular, the partial action of $H$ on $A$ is not a global action.

\esp
\begin{ex} \label{AcParc1} \emph{\cite[Example 6.1]{DFP}}
Let $B$ be an algebra and let
$$A = B \times B \times B = Be_1 \oplus Be_2 \oplus Be_3,$$
where $e_1 = (1_B, 0, 0)$, $e_2 = (0, 1_B, 0)$ and $e_3 = (0, 0, 1_B)$. If $G = \{ 1_G, g, g^2, g^3 \}$ is the cyclic group of order $4$, then $H = \kk G$ acts partially on $A$ by
$$g \cdot e_1 = 0 \, , \quad g \cdot e_2 = e_1 \, , \quad g \cdot e_3 = e_2,$$
$$g^2 \cdot e_1 = e_3 \, , \quad g^2 \cdot e_2 = 0 \, , \quad g^2 \cdot e_3 = e_1,$$
$$g^3 \cdot e_1 = e_2 \, , \quad g^3 \cdot e_2 = e_3 \, , \quad g^3 \cdot e_3 = 0.$$
\end{ex}
\esp


In the study of partial $H$-module algebras an important subalgebra is the so called \emph{invariant subalgebra}.

\esp
\begin{defi}
Let $A$ be a partial $H$-module algebra. The subalgebra of the invariant elements of $A$ is defined as the set
$$A^{\pH} = \conj{a \in A}{h \cdot a = a(h \cdot 1_A), \, \forall \, h \in H}.$$
\end{defi}
\esp

It is easy to see that $A^{\pH}$ is actually a subalgebra of $A$. When the action of $H$ on $A$ is global
then we denote the invariant subalgebra by $A^H$. If it is the case, we have
$$A^H = \conj{a \in A}{h \cdot a = \varepsilon(h) a, \, \forall \, h \in H}.$$

\esp
\begin{prop} \label{inverso}
Let $A$ be a partial $H$-module algebra. If $a \in A^{\pH}$ is invertible in $A$, then $a^{-1} \in A^{\pH}$.
In particular, we have $J(A)\cap A^{\pH} \subseteq J(A^{\pH})$.
\end{prop}

\begin{proof}
For any $h \in H$, we have
\begin{eqnarray*}
h \cdot a^{-1} & = & a^{-1} \sum{a (h_{1} \cdot 1_A) (h_2 \cdot a^{-1})} \,\, = \,\, a^{-1} \sum{(h_1 \cdot a)(h_2 \cdot a^{-1})} \\
&=& a^{-1} (h \cdot (a a^{-1})) \,\, = \,\, a^{-1} (h \cdot 1_A)
\end{eqnarray*}
and so $a^{-1} \in A^{\pH}$. As a consequence, $I := J(A) \cap A^{\pH} \,$ is an ideal of $A^{\pH}$ such that $1-I$ is contained in the set of units of $A^{\pH}$. Hence $J(A)\cap A^{\pH} \, \subseteq J(A^{\pH})$.
\end{proof}
\esp

Dualizing the concept of (left) partial $H$-module algebra we obtain the definition of (right) partial $H$-comodule algebra as follows.

\esp
\begin{defi} \label{defHcomodAlgPar}
A (right) partial coaction of a Hopf algebra $H$ on an algebra $R$ is a $\kk$-linear map $\rho: R \to R \ts H$, $x \mapsto \sum x_0 \ts x_1$ such that, for any $x,y \in R$,
\begin{itemize}
\item[\itm{(PC1)}] $(id_R \otimes \varepsilon)(\rho(x)) = x$
\item[\itm{(PC2)}] $\rho(xy) = \rho(x) \rho(y)$
\item[\itm{(PC3)}] $(\rho \otimes id_H)(\rho(x)) = [\rho(1_R) \otimes 1_H][(id_R \otimes \Delta)(\rho(x))]$.
\end{itemize}
In this case, $R$ is called (right) partial $H$-comodule algebra.
\end{defi}
\esp

In the Sweedler's notation,  the above conditions can be written as follows
\begin{itemize}
\item[{\itm{(PC1)}}] $\sum x_0 \varepsilon(x_1) = x$
\item[{\itm{(PC2)}}] $\sum (xy)_0 \ts (xy)_1 = \sum x_0 y_0 \ts x_1 y_1$
\item[{\itm{(PC3)}}] $\sum x_{00} \ts x_{01} \ts x_1 = \sum 1_0 x_0 \ts 1_1 x_{11} \ts x_{12}$.
\end{itemize}
for any $x,y \in R$.

Analogous to the case of partial actions, if $R$ is a partial $H$-comodule algebra and
$$\rho(1_R) = 1_R \ts 1_H,$$
then $R$ is an {\it $H$-comodule algebra} (see \cite[Definition 6.2.1]{DNR}). In this case $\rho$ is called {\it global coaction}. Also, it is immediate that any $H$-comodule algebra is a partial $H$-comodule algebra.

\esp
\begin{defi}
Let $R$ be a partial $H$-comodule algebra. The subalgebra of the coinvariant elements of $R$ is defined as the set
$$R^{\pcoH} = \conj{x \in R}{\rho(x) = (x \otimes 1_H) \rho(1_R) = \sum x 1_0 \ts 1_1}.$$
\end{defi}
\esp

The fact that $R^{\pcoH}$ is a subalgebra of $R$ is an immediate consequence of $\itm{(PC2)}$.When the coaction is
global we denote the coinvariant subalgebra by $R^{coH}$, that is
$$R^{coH} = \conj{x \in R}{\rho(x) = (x \otimes 1_H)}.$$

In the same sense of Proposition \ref{inverso}, we can prove the following result.

\esp
\begin{prop} \label{coinverso}
Let $R$ be a partial $H$-comodule algebra. If $x \in R^{\pcoH}$ is invertible in $R$, then $x^{-1} \in R^{\pcoH}$.
In particular,  we have $J(R)\cap R^{\pcoH} \subseteq J(R^{\pcoH})$.
\end{prop}

\begin{proof}
In fact,
\begin{eqnarray*}
\rho(x^{-1}) & = & (1_R \ts 1_H) \rho(1_R x^{-1}) \,\, = \,\, (x^{-1} \ts 1_H)(x \ts 1_H) \rho(1_R) \rho(x^{-1}) \\
& = & (x^{-1} \ts 1_H) \rho(x) \rho(x^{-1}) \,\, = \,\, (x^{-1} \ts 1_H) \rho(x x^{-1}) \,\, = \,\, (x^{-1} \ts 1_H) \rho(1_R).
\end{eqnarray*}
The rest is similar to the proof of Proposition \ref{inverso}.
\end{proof}
\esp

We finish this section recalling the definition of the partial smash product. Let $A$ be a partial $H$-module algebra. On the vector space $A \otimes H$ we consider the following multiplication:
$$(a \ts h)(b \ts g) = \sum{a(h_{1} \cdot b) \ts h_{2}g}, \quad \forall \, a,b \in A, \, \forall \, h,g \in H.$$
With this multiplication (and the usual addition) $A \ts H$ becomes an associative algebra. This algebra is denoted by $A \# H$ and their (generators) elements by $a \# h$ instead of $a \ts h$. In general, $A \# H$ has not identity element unless the action of $H$ on $A$ is global (see \cite[Proposition 6.1.7]{DNR}). Therefore we consider the following subspace of $A \# H$:
$$\psmash{A}{H} := (A \# H)(1_A \# 1_H),$$
generated by elements of the form
$$\psmash{a}{h} := (a \# h)(1_A \# 1_H) = \sum a(h_1 \cdot 1_A) \# h_2.$$
Thus, $\psmash{A}{H}$ is an associative algebra with identity element $\psmash{1_A}{1_H} = 1_A \# 1_H$.

\esp
\begin{defi}
Let $A$ be a partial $H$-module algebra. The algebra $\psmash{A}{H}$ is called the partial smash product of $A$ by $H$.
\end{defi}
\esp

Clearly, if $A$ is an $H$-module algebra (global action) then $\psmash{A}{H} = A \# H$. Moreover,
if $A$ is a partial $H$-module algebra then there is a natural monomorphism of algebras given by
\begin{eqnarray*}
\funn{f}{A}{\psmash{A}{H} \subseteq A \# H}{a}{\psmash{a}{1_H} = a \# 1_H,}
\end{eqnarray*}
and we will identify $A$ with the subalgebra $f(A) = \psmash{A}{1_H} = A \# 1_H$ of $\psmash{A}{H}$.


We further note that the partial smash product $\psmash{A}{H}$ has a structure of (right) $H$-comodule algebra (global coaction) given by
\begin{eqnarray*}
\funn{\rho}{\psmash{A}{H}}{\psmash{A}{H} \otimes H}{\psmash{a}{h}}{\psmash{a}{h_1} \ts h_2,}
\end{eqnarray*}
and that $\paren{\psmash{A}{H}}^{coH} = A$.

The next remark will be useful later.

\esp
\begin{remark} \label{JacA}
Suppose that $A$ is a partial $H$-module algebra. Then the above observation and \emph{Proposition \ref{coinverso}} give us the following interesting result
$$J\!\ppsmash{A}{H} \cap A \,\, \subseteq \,\, J(A).$$
\end{remark}
\esp

It is well known that if $H$ is finite-dimensional, then its dual $\dual{H}$ is a Hopf algebra. In this case, $\psmash{A}{H}$ has a structure of
left $\dual{H}$-module algebra induced by its right $H$-comodule algebra structure, where the (global) action is given by
$$\varphi \acind \paren{\psmash{a}{h}} := \sum \psmash{a}{h_1 \varphi(h_2)}, \qquad a \in A, \, h \in H, \, \varphi \in \dual{H}.$$
Moreover, $\paren{\psmash{A}{H}}^{\dual{H}} = \paren{\psmash{A}{H}}^{coH} = A$ (see \cite[Proposition 6.2.4]{DNR}).
\esp


\section{On $H$-stable ideals and partial $\modpar{A}{H}$-modules} \label{IdHesteAHmodpar}

Hereafter, unless mentioned otherwise, $A$ will denote a partial $H$-module algebra. In \cite{F} and \cite{Z} the authors investigated the concepts of $H$-stable ideal of $A$ and $\modpar{A}{H}$-module in the study of $H$-radicals when the action of $H$ on $A$ was global. These concepts can be adapted for the case of partial actions as we will see later. In the next section, we will explore these concepts in the study of the $H$-prime and $H$-Jacobson radicals of $A$ in the setting of partial actions. In particular, we will establish a link between the $H$-prime radical of $A$ and the $\dual{H}$-prime radical of $\psmash{A}{H}$ and also a link between the $H$-Jacobson radical of $A$ and the $\dual{H}$-Jacobson radical of $\psmash{A}{H}$, when $H$ is finite-dimensional.

We begin with the concept of $H$-stable ideal. The aim here is to establish a link between the set of all $H$-stable ideals of $A$ and the set of all ideals of $\psmash{A}{H}$ which are $H$-subcomodules (Theorem \ref{idProdSmashHcom}). When $H$ is finite-dimensional this last set is just the set of all $\dual{H}$-stable ideals of $\psmash{A}{H}$ (Corollary \ref{idHest2}).

\esp
\begin{defi} \label{defidHest}
Let $A$ be a partial $H$-module algebra. An ideal (resp. left, right ideal) $I$ of $A$ is said to be $H$-stable if $H \cdot I \subseteq I$.
\end{defi}

\vspace{0.1cm}

\begin{remark} \label{obsHestglob}
If the action of $H$ on $A$ is a global one, then the $H$-stable ideals of $A$ are exactly the ideals of $A$ which are $H$-submodules.
\end{remark}
\esp


Let $I$ be an $H$-stable ideal of $A$. Then the partial action of $H$ on $A$ induces a partial action of $H$ on the quotient $\quoc{A}{I}$ given by
$$h \cdot (a + I) := (h \cdot a) + I \, , \qquad a \in A, \, h \in H.$$
Moreover, we have the following algebra isomorphism:
$$\quoc{\paren{\psmash{A}{H}}}{\paren{\psmash{I}{H}}} \cong \psmash{(\quoc{A}{I})}{H} \, ,$$
where $\psmash{I}{H} := \conj{\sum \psmash{x_i}{h_i}}{x_i \in I, \, h_i \in H}$. As a consequence we have the following result.

\esp
\begin{prop} \label{ProdSubdirSmash}
Suppose that the partial $H$-module algebra $A$ is a subdirect product of partial $H$-module algebras $A_{\alpha} \cong A/I_{\alpha}$, where $\{ I_{\alpha} \}$ is a family of $H$-stable ideals of $A$ such that $\bigcap I_{\alpha} = 0$. Then $R := \psmash{A}{H}$ is a subdirect product of the algebras $R_{\alpha} := \psmash{A_{\alpha}}{H}$.
\end{prop}

\begin{proof}
As above, for every $\alpha$, there is an algebra isomorphism $R/\paren{\psmash{I_{\alpha}}{H}} \cong R_{\alpha}$, thus it is enough to show that $\bigcap \paren{\psmash{I_{\alpha}}{H}} = 0$. On the other hand, for every $\alpha$ we have that $\psmash{I_{\alpha}}{H} \subseteq I_{\alpha} \# H$. Thus, we will have achieved our goal if we show that $\bigcap (I_{\alpha} \# H) = 0$.

Suppose that $u \in \bigcap (I_{\alpha} \# H)$ and write $u = \sum x_i \# h_i$, where $\{ x_i \} \subseteq A$ and $\{ h_i \} \subseteq H$, with $\{ h_i \}$ linearly independent. Then $x_i \in I_{\alpha}$ for every $i$ and every $\alpha$. Thus, $x_i \in \bigcap I_{\alpha} = 0$, for every $i$, and so $u = 0$. Hence $\bigcap (I_{\alpha} \# H) = 0$, as desired.
\end{proof}
\esp

The above proposition will be useful in the study of the semiprimitivity of the partial smash product because a subdirect product of semiprimitive algebras is also a semiprimitive algebra (see \cite[Proposition 2.3.4]{GW}).


Given a subspace $X \subseteq A$, we will denote by
$$(X:H) := \conj{x \in X}{h \cdot x \in X, \, \forall \, h \in H}.$$

\esp
\begin{prop} \label{maioridealHest}
Let $I$ be an ideal of $A$. Then $(I:H)$ is the largest $H$-stable ideal of $A$ contained in $I$. In particular, $I$ is $H$-stable if and only if $(I:H) = I$.
\end{prop}

\begin{proof}
For every $x \in (I:H)$, $a \in A$ and $h \in H$, we have
$$h \cdot (xa) = \sum (h_1 \cdot x)(h_2 \cdot a) \, \in IA \subseteq I,$$
so $xa \in (I:H)$. Analogously $A(I:H) \subseteq (I:H)$ and therefore, $(I:H)$ is an ideal of $A$.

Moreover, if $x \in (I:H)$ and $g \in H$ then
$$h \cdot (g \cdot x) = \sum (h_1 \cdot 1_A)((h_2 g) \cdot x) \, \in AI \subseteq I, \quad \forall \, h \in H,$$
so $g \cdot x \in (I:H)$. Thus $(I:H)$ is $H$-stable.

Now we observe that $(I:H)$ contains every $H$-stable ideal of $A$ which is contained in $I$ and so the proof is complete.
\end{proof}
\esp


\begin{prop} \label{idHest1}
If $\mc{I}$ is an ideal of $\psmash{A}{H}$, then $\mc{I} \cap A$ is an $H$-stable ideal of $A$.
\end{prop}

\begin{proof}
Clearly $\mc{I} \cap A$ is an ideal of $A$. Moreover, for any $x \in \mc{I} \cap A$ and $h \in H$,
\begin{eqnarray*}
\psmash{h \cdot x}{1_H} &=& \sum \psmash{h_1 \cdot x}{\varepsilon(h_2) 1_H} \\
&=& \sum \psmash{(h_1 \cdot x)(h_2 \cdot 1_A)}{h_3 S(h_4)} \\
&=& \sum (\psmash{h_1 \cdot x}{h_2})(\psmash{1_A}{S(h_3)}) \\
&=& \sum (\psmash{1_A}{h_1})(\psmash{x}{1_H})(\psmash{1_A}{S(h_2)}) \; \in \, \mc{I},
\end{eqnarray*}
so $h \cdot x \, \in \mc{I} \cap A$. Hence $\mc{I} \cap A$ is an $H$-stable ideal of $A$.
\end{proof}
\esp

The next result generalizes, for the setting of partial actions, a known result about global actions (which is a consequence of \cite[Lemma 1.3]{MS}, because if the action of $H$ on $A$ is global, then $A \subseteq A \# H$ is a faithfully flat $H$-Galois extension).

\esp
\begin{teo} \label{idProdSmashHcom}
Let $H$ be a Hopf algebra and let $A$ be a partial $H$-module algebra. Then there exists an one-to-one correspondence between the sets
$$\xymatrix@1{
\{ H \text{-stable ideals of } A \} \ar@<0.5ex>[r]^-{\Phi} & \{ \text{Ideals of } \psmash{A}{H} \text{ which are } H \text{-subcomodules} \} \ar@<0.5ex>[l]^-{\Psi}
}$$
given by $\Phi(I) = \psmash{I}{H}$, where $I$ is an $H$-stable ideal of $A$, and $\Psi(\mc{I}) = \mc{I} \cap A$, where $\mc{I}$ is an ideal of $\psmash{A}{H}$ which is $H$-subcomodule. Moreover, $\Psi = \Phi^{-1}$ and these maps preserve inclusions, sums, (finite) products and intersections.
\end{teo}

\begin{proof}
We will denote by $R := \psmash{A}{H}$. Let $I$ be an $H$-stable ideal of $A$ and $\mc{I}$ an ideal of $R$ which is $H$-subcomodule. For any $a,b \in A$, $x \in I$ and $h,g,k \in H$, we have
$$(\psmash{a}{h})(\psmash{x}{g})(\psmash{b}{k}) \, = \, \sum \psmash{a(h_1 \cdot x)((h_2 g_1) \cdot b)}{h_3 g_2 k} \,\, \in \, \psmash{(A(H \cdot I)A)}{H} \subseteq \psmash{I}{H}.$$
Thus $R \! \ppsmash{I}{H} \!\! R \subseteq \psmash{I}{H}$ and so $\psmash{I}{H}$ is an ideal of $R$. Moreover, by Proposition \ref{idHest1} and using that $\psmash{I}{H}$ is an $H$-subcomodule of $R$ we can deduce that $\Phi$ and $\Psi$ are well defined.

Now consider a basis $\{ h_i \}$ of $H$ containing $1_H$. If $u \in \psmash{I}{H} \cap A \subseteq (I \# H) \cap A$, write $u = \sum x_i \# h_i$ with $\{ x_i \} \subseteq A$. Since $\{ h_i \}$ is a basis of $H$,
it follows that $\{ x_i \} \subseteq I$. On the other hand, $1_H \in \{ h_i \}$ and $u \in A$, so $x_i = 0$ if $h_i \neq 1_H$ and therefore $u \in I$. This shows that $\psmash{I}{H} \cap A \subseteq I$. The other inclusion is clear and so it follows that
$\Psi(\Phi(I)) = I$.

To prove that $\Phi (\Psi (\mc{I})) = \mc{I}$, fix  $\sum \psmash{x_i}{h_i} \in \mc{I}$. Then we have
\begin{eqnarray*}
\sum \psmash{x_i}{h_i} & = & \sum \big( \psmash{x_i(h_{i1} \cdot 1_A)}{1_H} \big) \paren{\psmash{1_A}{h_{i2}}} \\
& = & \sum \paren{\psmash{x_i}{h_{i11}}} \big( \psmash{1_A}{S(h_{i12})} \big) \paren{\psmash{1_A}{h_{i2}}}.
\end{eqnarray*}
Since $\mc{I}$ is an $H$-subcomodule of $R$, $\sum \psmash{x_i}{h_{i11}} \ts h_{i12} \ts h_{i2} \; \in \, \mc{I} \ts H \ts H$, so $\sum \psmash{x_i}{h_i} = \sum \big( \psmash{x_i(h_{i1} \cdot 1_A)}{1_H} \big) \paren{\psmash{1_A}{h_{i2}}} \; \in \, (\mc{I} \cap A) \paren{\psmash{A}{H}} = \psmash{(\mc{I} \cap A)}{H}.$
This shows that $\mc{I} \subseteq \psmash{(\mc{I} \cap A)}{H}$ and the other inclusion is clear. Hence $\mc{I} = \psmash{(\mc{I} \cap A)}{H} = \Phi(\Psi(\mc{I}))$ and so $\Phi$ and $\Psi$ are bijections with $\Psi = \Phi^{-1}$.

Now we need to prove that $\Phi$ and $\Psi$ preserve inclusions, sums, (finite) products and intersections. For inclusions, this statement is clear from definitions of $\Phi$ and $\Psi$. With regard to the others, firstly note that sums, finite products and intersections of $H$-stable ideals of $A$ (resp. $H$-subcomodules of $\psmash{A}{H}$) are $H$-stable ideals  of $A$ (resp. $H$-subcomodules of $\psmash{A}{H}$).   Now fix a family $\{ I_{\alpha} \}_{\alpha \in \Lambda}$ of $H$-stable ideals of $A$ and a family $\{ \mc{I}_{\beta} \}_{\beta \in \Gamma}$ of ideals of $\psmash{A}{H}$ which are $H$-subcomodules. Clearly
$$\Phi \paren{\sum I_{\alpha}} = \psmash{ \paren{{\textstyle \sum} I_{\alpha}}}{H} = \sum \psmash{I_{\alpha}}{H} = \sum \Phi(I_{\alpha})$$
and  consequently,
$$\Psi \paren{\sum \mc{I}_{\beta}} = \Psi \paren{\sum \Phi(\Psi(\mc{I}_{\beta}))} = \Psi \paren{\Phi \paren{\sum \Psi(\mc{I}_{\beta})}} = \sum \Psi(\mc{I}_{\beta}).$$
Therefore $\Phi$ and $\Psi$ preserve sums.

Now observe that if $I$ is a $H$-stable ideal of $A$, then $\psmash{I}{H} = IR$ and so $IR$ is an ideal of $R$. In fact, the inclusion $IR \subseteq \psmash{I}{H}$ is clear. On the other hand, for any $x \in I$ and $h \in H$, $\psmash{x}{h} = (\psmash{x}{1_H})(\psmash{1_A}{h}) \; \in IR$ and also $\psmash{I}{H} \subseteq IR$. Thus, for any $\alpha_1, \alpha_2 \in \Lambda$,
\begin{eqnarray*}
\Phi(I_{\alpha_1} I_{\alpha_2}) & = & \psmash{(I_{\alpha_1} I_{\alpha_2})}{H} \, = \, (I_{\alpha_1} I_{\alpha_2}) R \, = \, I_{\alpha_1} (I_{\alpha_2} R) \, = \, I_{\alpha_1} (R I_{\alpha_2} R) \\
& = & (I_{\alpha_1} R)(I_{\alpha_2} R) \, = \, (\psmash{I_{\alpha_1}}{H})(\psmash{I_{\alpha_2}}{H}) \, = \, \Phi(I_{\alpha_1}) \Phi(I_{\alpha_2})
\end{eqnarray*}
Consequently, for any $\beta_1, \beta_2 \in \Gamma$,
$$\Psi(\mc{I}_{\beta_1} \mc{I}_{\beta_2}) \, = \, \Psi \Big( \Phi \big( \Psi(\mc{I}_{\beta_1}) \big) \, \Phi \big( \Psi(\mc{I}_{\beta_2}) \big) \Big) \, = \, \Psi \Big( \Phi \Big( \Psi(\mc{I}_{\beta_1}) \, \Psi(\mc{I}_{\beta_2}) \Big) \Big) \, = \,  \Psi(\mc{I}_{\beta_1}) \Psi(\mc{I}_{\beta_2}).$$
Hence, for finite products of $H$-stable ideals of $A$ or ideal of $R$ which are $H$-subcomodules, the result follows by induction.

Finally, we show that $\Phi$ and $\Psi$ preserve intersections. Clearly
$$\Psi \paren{\bigcap \mc{I}_{\beta}} = \paren{\bigcap \mc{I}_{\beta}} \cap A = \bigcap (\mc{I}_{\beta} \cap A) = \bigcap \Psi(\mc{I}_{\beta}),$$
and it follows that
$$\Phi \paren{\bigcap I_{\alpha}} = \Phi \paren{\bigcap \Psi(\Phi(I_{\alpha}))} = \Phi \paren{\Psi \paren{\bigcap \Phi(I_{\alpha})}} = \bigcap \Phi(I_{\alpha}).$$
The proof is complete now.
\end{proof}
\esp

If $H$ is finite-dimensional, since the action of $\dual{H}$ on $\psmash{A}{H}$ is global, then a subspace of $\psmash{A}{H}$ is an $H$-subcomodule if and only if it is an $\dual{H}$-submodule (see \cite[Lemma 1.6.4]{M}). Thus, our next result follows by Remark \ref{obsHestglob} .

\esp
\begin{cor} \label{idHest2}
Let $H$ be a finite-dimensional Hopf algebra and let $A$ be a partial $H$-module algebra. There exists an one-to-one correspondence between the sets
$$\xymatrix@1{
\{ H \text{-stable ideals of } A \} \ar@<0.5ex>[r]^-{\Phi} & \{ \dual{H} \text{-stable ideals of } \psmash{A}{H} \} \ar@<0.5ex>[l]^-{\Psi}
}$$
given by $\Phi(I) = \psmash{I}{H}$, where $I$ is an $H$-stable ideal of $A$, and $\Psi(\mc{I}) = \mc{I} \cap A$, where $\mc{I}$ is an $\dual{H}$-stable ideal of $\psmash{A}{H}$. Moreover $\Psi = \Phi^{-1}$ and these bijections preserve inclusions, sums, (finite) products and intersections.
\end{cor}
\esp


Now we present the definition of partial $\modpar{A}{H}$-modules. This concept is, in a certain sense, a generalization of the concept of $A$-module and it allows us to better understand the relation between $A$ and $\psmash{A}{H}$.

\esp
\begin{defi} \label{defAHmod}
Let $A$ be a (left) partial $H$-module algebra. A right (resp. left) partial $\modpar{A}{H}$-module is a right (resp. left) $A$-module $M$ with a $\kk$-linear map $M \ts H \rightarrow M$, $m \tssp h \mapsto m \acd h$ (resp. $H \ts M \rightarrow M$, $h \tssp m \mapsto h \ace m$) such that, for any $m \in M$, $a \in A$ and $h,g \in H$, we have
\begin{itemize}
\item[\itm{(PM1)}] $m \acd 1_H = m \qquad (\text{resp.} \;\; 1_H \ace m = m)$
\item[\itm{(PM2)}] $((m \acd h)a) \acd g = \sum (m(h_1 \cdot a)) \acd (h_2 g) \qquad (\text{resp.} \;\; h \ace (a(g \ace m)) = \sum (h_1 \cdot a)((h_2 g) \ace m))$.
\end{itemize}
\end{defi}
\esp

Clearly, the conditions $\itm{(PM1)}$ and $\itm{(PM2)}$ are equivalent to $\itm{(PM1)}$ and the following two conditions: for any $m \in M$, $a \in A$ and $h,g \in H$,
\begin{itemize}
\item[\itm{(PM3)}] $(m \acd h)a = \sum (m(h_1 \cdot a)) \acd h_2 \qquad(\text{resp.} \;\; h \ace (am)) = \sum (h_1 \cdot a)(h_2 \ace m))$
\item[\itm{(PM4)}] $(m \acd h) \acd g = \sum (m(h_1 \cdot 1_A)) \acd (h_2 g) \qquad (\text{resp.} \;\; h \ace (g \ace m) = \sum (h_1 \cdot 1_A)((h_2 g) \ace m))$.
\end{itemize}

\esp
\begin{remark}
It is worth noting here that the apparent lack of symmetry between the definitions of right and left partial $\modpar{A}{H}$-module occurs because $H$ acts on $A$ by the left. For this reason, in some times, the argumentations in the proofs of our results about left and right partial $\modpar{A}{H}$-modules are quite different. For example, if $V$ is an $A$-submodule of a right partial $\modpar{A}{H}$-module $M$, then the partial $\modpar{A}{H}$-submodule of $M$ generated by $V$ is $(VA) \acd H = V \acd H$, but if $M$ is a left partial $\modpar{A}{H}$-module then the partial $\modpar{A}{H}$-submodule of $M$ generated by $V$ is $A(H \ace V)$.
\end{remark}
\esp

The Proposition \ref{AHmod1} will justify our interest in these objects and provides examples of right partial $\modpar{A}{H}$-modules. If $A$ is a (left) partial $H$-module algebra, then $A$ itself is a left partial $\modpar{A}{H}$-module. More generally, if $I$ is a $H$-stable left ideal of $A$ then $I$ and $A/I$ are left partial $\modpar{A}{H}$-modules. In the next section we will see how it is possible to build a partial $\modpar{A}{H}$-module from a given $A$-module.


As usual, we denote by $\ann_A(M)$ (or simply $\ann(M)$) the annihilator of $M$. We recall that $M$ is called faithful if $\ann(M) = 0$. The next result says about the annihilator of a partial $\modpar{A}{H}$-module.

\esp
\begin{prop} \label{anulHideal}
If $M$ is a partial $\modpar{A}{H}$-module then $\ann(M)$ is an $H$-stable ideal of $A$.
\end{prop}

\begin{proof}
It is clear that $\ann(M)$ is an ideal of $A$. Moreover, for any $m \in M$, $x \in \ann(M)$ and $h \in H$, we have
\begin{eqnarray*}
m(h \cdot x) & = & \sum (m(h_1 \cdot x)) \acd (\varepsilon(h_2) 1_H) \\
& = & \sum (m(h_1 \cdot x)) \acd (h_2 S(h_3)) \\
& = & \sum ((m \acd h_1)x) \acd S(h_2) \; \in \; (Mx) \acd H = 0
\end{eqnarray*}
if $M$ is a right partial $\modpar{A}{H}$-module and also
\begin{eqnarray*}
(h \cdot x)m & = & \sum (h_1 \cdot x)((\varepsilon(h_2) 1_H) \ace m) \\
& = & \sum (h_1 \cdot x)((h_2 S(h_3)) \ace m) \\
& = & \sum h_1 \ace (x(S(h_2) \ace m)) \; \in \; H \ace (xM) = 0
\end{eqnarray*}
if $M$ is a left partial $\modpar{A}{H}$-module. Thus, in any case, $h \cdot x \in \ann(M)$. Hence $\ann(M)$ is $H$-stable.
\end{proof}
\esp

It is clear that if $M$ is a right (resp. left) partial $\modpar{A}{H}$-module and $I$ is an $H$-stable ideal of $A$ such that $I \subseteq \ann(M)$, then $M$ has a natural structure of right (resp. left) partial $\modpar{A/I}{H}$-module.

\esp
\begin{prop} \label{AHmod1}
If $M$ is a right (resp. left) partial $\modpar{A}{H}$-module then $M$ is a right (resp. left) $\psmash{A}{H}$-module, with action given by
$$m(\psmash{a}{h}) := (ma) \acd h \quad \, (\text{resp.} \;\; (\psmash{a}{h})m := a(h \ace m)), \quad \, m \in M, \, a \in A, \, h \in H.$$
Conversely, if $M$ is a right (resp. left) $\psmash{A}{H}$-module then $M$ is a right (resp. left) partial $\modpar{A}{H}$-module, with actions given by
\begin{eqnarray*}
ma := m(\psmash{a}{1_H}) &  & (\text{resp. } am := (\psmash{a}{1_H})m), \quad \, m \in M, \, a \in A \\
m \acd h := m(\psmash{1_A}{h}) &  & (\text{resp. } h \ace m := (\psmash{1_A}{h})m), \quad \, m \in M, \, h \in H.
\end{eqnarray*}
Moreover, in both cases we have that $\ann_A(M) \, = \, \ann_{\psmash{A}{H}}(M) \, \cap \, A$.
\end{prop}

\begin{proof}
Suppose that $M$ is a right partial $\modpar{A}{H}$-module. The map
\begin{eqnarray*}
\funn{\zeta}{M \ts \psmash{A}{H}}{M}{m \tssp \psmash{a}{h}}{(ma) \acd h}
\end{eqnarray*}
is well defined and $\kk$-linear. In fact, if $\varsigma$ and $\xi$ are the actions of $A$ and $H$ on $M$, respectively, then the $\kk$-linear map
$$\kappa : M \ts A \# H \cong M \ts A \ts H \xrightarrow{\varsigma \tsspp id_H} M \ts H \stackrel{\xi}{\longrightarrow} M$$
is such that, for any $m \in M$, $a \in A$ and $h \in H$,
$$\kappa \paren{m \tssp \psmash{a}{h}} \, = \, \kappa \big( m \tssp {\textstyle \sum} a(h_1 \cdot 1_A) \# h_2 \big) \, = \, \sum (ma(h_1 \cdot 1_A)) \acd h_2 \, = \, ((ma) \acd h)1_A \, = \, (ma) \acd h.$$
Thus $\zeta$ is the restriction of $\kappa$ to the subspace $M \ts \psmash{A}{H}$, and therefore $\zeta$ is a well defined $\kk$-linear map. Now, for any $m \in M$, $a,b \in A$ and $h,g \in H$, it is clear that $m \! \ppsmash{1_A}{1_H} = m$ and also we have
\begin{eqnarray*}
\paren{m \ppsmash{a}{h}} \ppsmash{b}{g} & = & (((ma) \acd h)b) \acd g \\
& = & \sum (ma(h_1 \cdot b)) \acd (h_2 g) \\
& = & m \paren{\sum \psmash{a(h_1 \cdot b)}{h_2 g}} \\
& = & m \paren{ \paren{\psmash{a}{h}} \paren{\psmash{b}{g}}},
\end{eqnarray*}
so $\zeta$ defines a right $\psmash{A}{H}$-module structure on $M$.

Conversely, suppose that $M$ is a right $\psmash{A}{H}$-module. Again we need verify that the maps
\begin{eqnarray*}
\ffun{M \ts A}{M}{m \tssp a}{m(\psmash{a}{1_H})}
\end{eqnarray*}
and
\begin{eqnarray*}
\ffun{M \ts H}{M}{m \tssp h}{m(\psmash{1_A}{h})}
\end{eqnarray*}
are well defined and $\kk$-linear. The first one follows by the algebra inclusion $A \subseteq \psmash{A}{H}$.  The second one is exactly the composition of the $\kk$-linear maps
\begin{eqnarray*}
M \ts H & \longrightarrow & M \ts A \ts H \xrightarrow{id_M \tsspp \pi} M \ts \psmash{A}{H} \stackrel{\zeta}{\longrightarrow} M \\
m \tssp h & \mapsto & m \tssp 1_A \tssp h
\end{eqnarray*}
where $\fun{\pi}{A \ts H}{\psmash{A}{H}}$, $a \tssp h \mapsto (a \# h)(1_A \# 1_H) = \psmash{a}{h}$, is the canonical projection and $\zeta$ is the action of $\psmash{A}{H}$ on $M$. The conditions $\itm{(PM1)}$ and $\itm{(PM2)}$ of the Definition \ref{defAHmod} follow from the $\psmash{A}{H}$-module structure of $M$: for any $m \in M$, $a \in A$ and $h,g \in H$, it is clear that $m \acd 1_H = m$ and also
\begin{eqnarray*}
\big( (m \acd h)a \big) \acd g & = & \Big( \paren{m \paren{\psmash{1_A}{h}}} \paren{\psmash{a}{1_H}} \Big) \paren{\psmash{1_A}{g}} \\
& = & m \paren{\paren{\psmash{1_A}{h}} \paren{\psmash{a}{1_H}} \paren{\psmash{1_A}{g}}} \\
& = & m \paren{\sum \psmash{h_1 \cdot a}{h_2 g}} \\
& = & \sum \big( m(h_1 \cdot a) \big) \acd (h_2 g)
\end{eqnarray*}
Hence $M$ is a right partial $\modpar{A}{H}$-module in this case.

To prove that  $M$ is a left $\modpar{A}{H}$-module if and only if $M$ is a left $\psmash{A}{H}$-module we proceed in the same way.
Finally, we observe that the equality $\ann_A(M) = \ann_{\psmash{A}{H}}(M) \cap A$ is clear.
\end{proof}
\esp

Note that the Proposition \ref{anulHideal} is also a consequence of the Proposition \ref{idHest1} and of the equality $\ann_A(M) \, = \, \ann_{\psmash{A}{H}}(M) \, \cap \, A$.
\esp


\section{The $H$-prime and the $H$-Jacobson radicals} \label{Hradicais}

In this section we study the concepts of $H$-prime and $H$-Jacobson radicals of a partial $H$-module algebra $A$. These structures are analogous to the prime and Jacobson radicals of any algebra, and they are useful to study the semiprimality and the semiprimitivity problems for the partial smash product.

\subsection{$H$-(semi)prime ideals and the $H$-prime radical}

We start by considering the (partial) $H$-prime radical. In \cite{MS} the authors studied the $H$-prime and the $H$-semiprime ideals of an $H$-module algebra $A$ in the context of the global actions. Such concepts can also be considered when the action of $H$ on $A$ is a partial action. In the sequel we introduce this new concepts with the purpose to define the $H$-prime radical of $A$, when $A$ is a partial $H$-module. After that, we establish a link between the $H$-prime radical of $A$ and the $\dual{H}$-prime radical of $\psmash{A}{H}$, when $H$ is finite-dimensional. Our approach for $H$-prime and $H$-semiprime ideals, in a natural way, closely follows the classical one which  can be found, for example, in \cite[Section 10]{L1} about prime and semiprime ideals. We will present here the sketches of our proofs for the convenience of the reader.

\esp
\begin{defi} \label{defHprimo}
An $H$-stable ideal $\mf{p}$ of $A$ is called $H$-prime if $\mf{p} \neq A$ and, for any $H$-stable ideals $I, J \subseteq A$,
$$IJ \subseteq \mf{p} \quad \Rightarrow \quad I \subseteq \mf{p} \quad \text{or} \quad J \subseteq \mf{p}.$$
\end{defi}
\esp

An example of $H$-prime ideal is given by a maximal $H$-stable ideal of $A$, whose existence is a consequence of Zorn's Lemma.
The next result gives other useful characterizations of $H$-prime ideals.

\esp
\begin{prop} \label{caracIdHprimos}
For any $H$-stable ideal $\mf{p} \varsubsetneq A$, the following statements are equivalent:
\begin{itemize}
\item[\itm{(1)}] $\mf{p}$ is $H$-prime;
\item[\itm{(2)}] For any $a, b \in A$, $A(H \cdot a)A(H \cdot b)A \subseteq \mf{p}$ implies that $a \in \mf{p}$ or $b \in \mf{p}$;
\item[\itm{(3)}] For any $a, b \in A$, $A(H \cdot a)A(H \cdot b) \subseteq \mf{p}$ implies that $a \in \mf{p}$ or $b \in \mf{p}$;
\item[\itm{(4)}] For any $a, b \in A$, $(H \cdot a)A(H \cdot b) \subseteq \mf{p}$ implies that $a \in \mf{p}$ or $b \in \mf{p}$;
\item[\itm{(5)}] For any $H$-stable left ideals $I, J \subseteq A$, $IJ \subseteq \mf{p}$ implies that $I \subseteq \mf{p}$ or $J \subseteq \mf{p}$;
\item[\itm{(5')}] For any $H$-stable right ideals $I, J \subseteq A$, $IJ \subseteq \mf{p}$ implies that $I \subseteq \mf{p}$ or $J \subseteq \mf{p}$.
\end{itemize}
\end{prop}

\begin{proof}
$\itm{(1)} \Rightarrow \itm{(2)}$ follows because $A(H \cdot a)A$ and $A(H \cdot b)A$ are $H$-stable ideals and the implications $\itm{(2)} \Rightarrow \itm{(3)} \Rightarrow \itm{(4)}$ are easy. For $\itm{(4)} \Rightarrow \itm{(5)}$  we assume that $I$ and $J$ are $H$-stable left ideals of $A$ such that $IJ \subseteq \mf{p}$, but $I \nsubseteq \mf{p}$. Fix an element $a \in \comp{I}{\mf{p}}$. For any $b \in J$,
$$(H \cdot a)A(H \cdot b) \subseteq I(AJ) \subseteq IJ \subseteq \mf{p},$$
so by $\itm{(4)}$, $b \in \mf{p}$. It follows that $J \subseteq \mf{p}$. The implication $\itm{(4)} \Rightarrow \itm{(5')}$ is analogous and the implications $\itm{(5)} \Rightarrow \itm{(1)}$ and $\itm{(5')} \Rightarrow \itm{(1)}$ are trivial.
\end{proof}

Following the terminology used in \cite{L1}, we present our next definition.

\esp

\begin{defi}
We say that a nonempty subset $\mpzc{M} \subseteq A$ is an $Hm$-system if for any $x,y \in \mpzc{M}$, we have
$$A(H \cdot x)A(H \cdot y) \cap \mpzc{M} \neq \emptyset.$$
\end{defi}
\esp

From the above definition we can give some more precise characterizations of $H$-prime ideal. The first one is an immediate consequence of Proposition \ref{caracIdHprimos} $\itm{(3)}$.

\esp
\begin{prop} \label{HprimeHmsist1}
An $H$-stable ideal $\mf{p}$ of $A$ is $H$-prime if and only if $A \backslash \mf{p}$ is an $Hm$-system.
\end{prop}

\vspace{0.1cm}

\begin{prop} \label{HprimeHmsist2}
Let $\mpzc{M} \subseteq A$ be an $Hm$-system. If $\mf{p}$ is an $H$-stable ideal of $A$ maximal with respect to the property that $\mf{p} \cap \mpzc{M} = \emptyset$, then $\mf{p}$ is $H$-prime.
\end{prop}

\begin{proof}
Let $I, J$ be $H$-stable ideals of $A$ such that $I,J \nsubseteq \mf{p}$. Then $\mf{p} \varsubsetneq \mf{p} + I$ and $\mf{p} \varsubsetneq \mf{p} + J$, and it follows that there exist $x,y \in A$ such that
$$x \in \mpzc{M} \cap (\mf{p} + I) \quad \text{and} \quad y \in \mpzc{M} \cap (\mf{p} + J).$$
Since $\mpzc{M}$ is an $Hm$-system, $A(H \cdot x)A(H \cdot y) \cap \mpzc{M} \neq \emptyset$. On the other hand, we also have
$$A(H \cdot x)A(H \cdot y) \,\, \subseteq \,\, A \big( H \cdot (\mf{p} + I) \big) A \big( H \cdot (\mf{p} + J) \big) \,\, \subseteq \,\, (\mf{p} + I)(\mf{p} + J) \,\, \subseteq \,\, \mf{p} + IJ,$$
which implies that $(\mf{p} + IJ) \cap \mpzc{M} \neq \emptyset$. From $\mf{p} \cap \mpzc{M} = \emptyset$, it follows that $IJ \nsubseteq \mf{p}$. Hence $\mf{p}$ is $H$-prime.
\end{proof}

\esp

\begin{defi} \label{defHradical}
Let $I$ be an $H$-stable ideal of $A$. We define the $H$-radical of $I$ as the set
$$\Hrz{I} := \conj{x \in A}{\text{every} \; Hm\text{-system containing} \; x \; \text{meets} \; I}.$$
In the particular case when $I = 0$, we will  call $P_H(A) := \Hrz{0}$ the $H$-prime radical of $A$.
\end{defi}
\esp

Clearly $I \subseteq \Hrz{I}$ for any $H$-stable ideal $I$ of $A$. We say that an ideal $I$ of $A$ is an {\it $H$-radical ideal} if $I = \Hrz{I}$.

It is not clear from the Definition \ref{defHradical} that $\Hrz{I}$ is an $H$-stable ideal of $A$. This will be shown in the next result.

\esp
\begin{prop} \label{caracHradical}
For any $H$-stable ideal $I$ of $A$, $\Hrz{I}$ is the intersection of all  $H$-prime ideals of $A$ containing $I$. In particular, $\Hrz{I}$ is an $H$-stable ideal of $A$.
\end{prop}

\begin{proof}
Let $\mf{p}$ be an $H$-prime ideal of $A$ containing $I$ and let $x \in \Hrz{I}$. Since $\comp{A}{\mf{p}}$ is an $Hm$-system (Proposition \ref{HprimeHmsist1}) which does not meet $I$ it follows that $x\in \mf{p}$. Thus, $\Hrz{I}$ is contained in the intersection of all  $H$-prime ideals of $A$ containing $I$.

Conversely, if $x \notin \Hrz{I}$ then there exists an $Hm$-system $\mpzc{M} \subseteq A$ such that $x \in \mpzc{M}$ and $\mpzc{M} \cap I = \emptyset$. By Zorn's Lemma, there exists an $H$-stable ideal $\mf{p}$ of $A$ containing $I$ which is maximal with respect to being $\mf{p} \cap \mpzc{M} = \emptyset$. By Proposition \ref{HprimeHmsist2}, $\mf{p}$ is an $H$-prime ideal (which contains $I$ and $x\not\in \mf{p}$). The proof is complete now.
\end{proof}
\esp


\begin{defi} \label{defHsemiprimo}
An $H$-stable ideal $\mf{s}$ of $A$ is called $H$-semiprime if $\mf{s} \neq A$ and, for any $H$-stable ideal $I \subseteq A$,
$$I^2 \subseteq \mf{s} \quad \Rightarrow \quad I \subseteq \mf{s}.$$
\end{defi}
\esp

Clearly every $H$-prime ideal is $H$-semiprime. Moreover, every intersection of $H$-semiprime ideals is also $H$-semiprime.

The proof of the next proposition is similar to that for Proposition \ref{caracIdHprimos}.

\esp
\begin{prop} \label{caracIdHsemiprimos1}
For any $H$-stable ideal $\mf{s} \varsubsetneq A$, the following statements are equivalent:
\begin{itemize}
\item[\itm{(1)}] $\mf{s}$ is $H$-semiprime;
\item[\itm{(2)}] For any $a \in A$, $A(H \cdot a)A(H \cdot a)A \subseteq \mf{s}$ implies that $a \in \mf{s}$;
\item[\itm{(3)}] For any $a \in A$, $A(H \cdot a)A(H \cdot a) \subseteq \mf{s}$ implies that $a \in \mf{s}$;
\item[\itm{(4)}] For any $a \in A$, $(H \cdot a)A(H \cdot a) \subseteq \mf{s}$ implies that $a \in \mf{s}$;
\item[\itm{(5)}] For any $H$-stable left ideal $I \subseteq A$, $I^2 \subseteq \mf{s}$ implies that $I \subseteq \mf{s}$;
\item[\itm{(5')}] For any $H$-stable right ideal $I \subseteq A$, $I^2 \subseteq \mf{s}$ implies that $I \subseteq \mf{s}$.
\end{itemize}
\end{prop}

\vspace{0.1cm}

\begin{defi}
We say that a nonempty subset $\mpzc{N} \subseteq A$ is an $Hn$-system, if for any $x \in \mpzc{N}$, we have
$$A(H \cdot x)A(H \cdot x) \cap \mpzc{N} \neq \emptyset.$$
\end{defi}
\esp

As an immediate consequence of the Proposition \ref{caracIdHsemiprimos1} $\itm{(3)}$ we also can characterize $H$-semiprime ideals by $Hn$-system.

\esp
\begin{prop} \label{HsemiprimeHnsist1}
An $H$-stable ideal $\mf{s}$ of $A$ is $H$-semiprime if and only if $A \backslash \mf{s}$ is an $Hn$-system.
\end{prop}
\esp

The next result can be proved in a similar way as Proposition \ref{HprimeHmsist2}.

\esp
\begin{prop} \label{HsemiprimeHnsist2}
Let $\mpzc{N} \subseteq A$ be an $Hn$-system. If $\mf{s}$ is an $H$-stable ideal of $A$ maximal with respect to the property that $\mf{s} \cap \mpzc{N} = \emptyset$, then $\mf{s}$ is $H$-semiprime.
\end{prop}
\esp

The next result establishes a link between $Hm$-systems and $Hn$-systems which allow us a better understanding of the relations between $H$-prime  and $H$-semiprime ideals.

\esp
\begin{prop} \label{HmsisteHnsist}
Let $\mpzc{N} \subseteq A$ be an $Hn$-system. For every $a \in \mpzc{N}$, there exists an $Hm$-system $\mpzc{M} \subseteq \mpzc{N}$ such that $a \in \mpzc{M}$. Equivalently, $\mpzc{N}$ is equal to the union of the $Hm$-systems contained in $\mpzc{N}$.
\end{prop}

\begin{proof}
Fix $a \in \mpzc{N}$ and define $\mpzc{M} = \{ a_1, a_2, ... \}$ inductively as follows:
$$a_1 = a, \quad a_2 \in A(H \cdot a_1)A(H \cdot a_1) \cap \mpzc{N}, \quad a_3 \, \in \, A(H \cdot a_2)A(H \cdot a_2) \cap \mpzc{N}, \,\,\, \ldots$$
To see that $\mpzc{M}$ is an $Hm$-system we first observe that, for any $t \geq 1$,
$$H \cdot a_{t+1} \subseteq H \cdot \big( A(H \cdot a_t)A(H \cdot a_t) \big) \subseteq A(H \cdot a_t),$$
which implies $H \cdot a_{t'} \subseteq A(H \cdot a_t)$ whenever $t \leq t'$. Therefore, for any $i,j \geq 1$, we have
$$\left\{ \begin{array}{rc}
a_{j+1} \, \in \, A(H \cdot a_j)A(H \cdot a_j) \, \subseteq \, A(H \cdot a_i)A(H \cdot a_j) & \, \text{if} \, \quad i \leq j, \\
a_{i+1} \: \in \: A(H \cdot a_i)A(H \cdot a_i) \: \subseteq \, A(H \cdot a_i)A(H \cdot a_j) & \, \text{if} \, \quad i \geq j,
\end{array} \right.$$
in particular,
$$A(H \cdot a_i)A(H \cdot a_j) \cap \mpzc{M} \neq \emptyset.$$
Hence $\mpzc{M}$ is an $Hm$-system such that $a \in \mpzc{M} \subseteq \mpzc{N}$.
\end{proof}

\esp

\begin{prop} \label{caracIdHsemiprimos2}
For any $H$-stable ideal $\mf{s} \varsubsetneq A$, the following statements are equivalent:
\begin{itemize}
\item[\itm{(1)}] $\mf{s}$ is $H$-semiprime;
\item[\itm{(2)}] $\mf{s}$ is an intersection of $H$-prime ideals;
\item[\itm{(3)}] $\mf{s} \; = \Hrz{\mf{s}}$.
\end{itemize}
\end{prop}

\begin{proof}
$\itm{(3)} \Rightarrow \itm{(2)}$ follows by Proposition \ref{caracHradical} and $\itm{(2)} \Rightarrow \itm{(1)}$ is clear. For $\itm{(1)} \Rightarrow \itm{(3)}$, it is enough to prove that $\Hrz{\mf{s}} \subseteq \mf{s}$. Thus, consider $a \notin \mf{s}$. Then $\comp{A}{\mf{s}}$ is an $Hn$-system containing $a$. By Proposition \ref{HmsisteHnsist}, there exists an $Hm$-system $\mpzc{M} \subseteq \comp{A}{\mf{s}}$ such that $a \in \mpzc{M}$ and so $a \notin \Hrz{\mf{s}}$.
\end{proof}
\esp

In particular, the $H$-semiprime ideals of a partial $H$-module algebra are precisely the $H$-radical ideals. The following result is immediate from Propositions \ref{caracHradical} and \ref{caracIdHsemiprimos2}.

\esp
\begin{cor} \label{HradmenorHsemiprimo}
For any $H$-stable ideal $I$ of $A$, $\Hrz{I}$ is the smallest $H$-semiprime ideal of $A$ containing $I$. In particular, $\Hrz{\Hrz{I}} = \Hrz{I}$.
\end{cor}

\vspace{0.1cm}

\begin{defi} \label{defHMAlgHprimoeHsemiprimo}
A partial $H$-module algebra is called $H$-prime (resp. $H$-semiprime) if $0$ is an $H$-prime (resp. $H$-semiprime) ideal.
\end{defi}
\esp

It is clear that an $H$-stable ideal $\mf{p}$ of $A$ is $H$-prime if and only if $A/\mf{p}$ is an $H$-prime partial $H$-module algebra.

\esp
\begin{prop} \label{caracHMAlgHsemiprimo}
For a partial $H$-module algebra $A$, the following statements are equivalent:
\begin{itemize}
\item[\itm{(1)}] $A$ is $H$-semiprime;
\item[\itm{(2)}] $P_H(A) = 0$;
\item[\itm{(3)}] $A$ has no nonzero nilpotent $H$-stable ideal;
\item[\itm{(4)}] $A$ has no nonzero nilpotent $H$-stable left ideal;
\item[\itm{(4')}] $A$ has no nonzero nilpotent $H$-stable right ideal;
\end{itemize}
\end{prop}

\begin{proof}
$\itm{(1)} \Leftrightarrow \itm{(2)}$ follows from Corollary \ref{HradmenorHsemiprimo}. For $\itm{(1)} \Rightarrow \itm{(4)}$, let $I$ be a nilpotent $H$-stable left ideal and let $n \geq 1$ be the smallest positive integer such that $I^n = 0$. If $n \geq 2$, then $(I^{n-1})^2 \subseteq I^n = 0$ which implies that $I^{n-1} = 0$, a contradiction. Thus $n = 1$ and $I = 0$. The implication $\itm{(1)} \Rightarrow \itm{(4')}$ is similar and the implications $\itm{(3)} \Rightarrow \itm{(1)}$, $\itm{(4)} \Rightarrow \itm{(3)}$ and $\itm{(4')} \Rightarrow \itm{(3)}$ are easy to see.
\end{proof}
\esp

The next proposition and its corollary establish a relation between the prime (resp. semiprime) ideals and the $H$-prime (resp. $H$-semiprime) ideals of a partial $H$-module algebra. Like in \cite{L1}, we use the notation $\rz{I}$ to indicate the (classical) radical of the ideal $I$ of $A$, which is the intersection of all the prime ideals of $A$ containing $I$. In the case which $I = 0$, we denote by $P(A) := \rz{0}$ the prime radical of $A$.

\esp
\begin{prop} \label{IdprimoIdHprimo}
If $\mc{P}$ is a prime (resp. semiprime) ideal of $A$, then $(\mc{P}:H)$ is $H$-prime (resp. $H$-semiprime). If $H$ is finite-dimensional, then the converse is true, that is, every $H$-prime (resp. $H$-semiprime) ideal of $A$ is of the form $(\mc{P}:H)$ for some prime (resp. semiprime) ideal $\mc{P}$ of $A$.
\end{prop}

\begin{proof}
Let $\mc{P}$ be a prime ideal of $A$. If $I$ and $J$ are $H$-stable ideals of $A$ such that $IJ \subseteq (\mc{P}:H) \subseteq \mc{P}$, then we have $I \subseteq \mc{P}$ or $J \subseteq \mc{P}$. The $H$-stability of $I$ and $J$ implies that $I \subseteq (\mc{P}:H)$ or $J \subseteq (\mc{P}:H)$. Hence $(\mc{P}:H)$ is $H$-prime. The ``semiprime'' case is analogous.

Now suppose that $H$ is finite-dimensional and let $\mf{p}$ be an $H$-prime ideal of $A$. Consider the following two family of ideals of $A$:
\begin{eqnarray*}
\mc{G} & = & \conj{J \ideal A}{J \text{ is } H \text{-stable and } J \nsubseteq \mf{p}} \\
\mc{F} & = & \conj{K \ideal A}{\mf{p}\subseteq K \text{ and } J \nsubseteq K, \, \forall \, J \in \mc{G}}.
\end{eqnarray*}
We will use the Zorn's Lemma to prove that $\mc{F}$ has a maximal element. Note that $\mc{F} \neq \emptyset$ since $\mf{p} \in \mc{F}$. Now, let $\{ K_{\lambda} \} \subseteq \mc{F}$ be a chain. Then $\mc{K} = \bigcup K_{\lambda}$ is an ideal of $A$ and $\mf{p} \subseteq \mc{K}$. We claim that $J \nsubseteq \mc{K}$ for every $J \in \mc{G}$. In fact, if $J \in \mc{G}$, then there exists $a \in J$ such that $a \notin \mf{p}$, so $A(H \cdot a)A \nsubseteq \mf{p}$ and $A(H \cdot a)A \in \mc{G}$. This implies that $A(H \cdot a)A \nsubseteq K_{\lambda}$ for every $\lambda$. Because $A(H \cdot a)A$ is a finitely generated ideal (since $H$ is finite-dimensional) and $\{ K_{\lambda} \}$ is a chain, we have that $A(H \cdot a)A \nsubseteq \mc{K}$. In particular $J \nsubseteq \mc{K}$ (note that $A(H \cdot a)A \subseteq J$ since $a \in J$ and $J$ is an $H$-stable ideal). It follows that $\mc{K} \in \mc{F}$ is an upper bound for $\{ K_{\lambda} \}$. By Zorn's Lemma, there exists a maximal element $\mc{P} \in \mc{F}$.

We claim that $\mc{P}$ is a prime ideal. In fact, suppose that $I$ and $J$ are ideals of $A$ not contained in $\mc{P}$. By maximality of $\mc{P}$, there exist $H$-stable ideals $K, L \in \mc{G}$ such that $K \subseteq \mc{P} + I$ and $L \subseteq \mc{P} + J$, and so $KL \subseteq \mc{P} + IJ$. Since $KL \in \mc{G}$ (because $\mf{p}$ is $H$-prime), it follows that $KL \nsubseteq \mc{P}$, and so $IJ \nsubseteq \mc{P}$. Hence $\mc{P}$ is a prime ideal.

Since $\mc{P} \in \mc{F}$ and $\mf{p}$ is $H$-stable, we have $\mf{p} \subseteq (\mc{P}:H)$. On the other hand, if $J$ is an $H$-stable ideal such that $\mf{p} \subsetneq J$ then $J \in \mc{G}$ and so $J \nsubseteq \mc{P}$. Hence $\mf{p} = (\mc{P}:H)$ and every $H$-prime ideal of $A$ is of the form $(\mc{P}:H)$ for some prime ideal $\mc{P}$ of $A$, when $H$ is finite-dimensional.

Finally, assume that  $\mf{s}$ is an $H$-semiprime ideal of $A$ (and $H$ is finite-dimensional). By Proposition \ref{caracIdHsemiprimos2}, $\mf{s} = \bigcap \mf{p}_{\alpha}$, where $\mf{p}_{\alpha}$ runs over all the $H$-prime ideals of $A$ containing $\mf{s}$. From the above part, for every $\alpha$, $\mf{p}_{\alpha} = (\mc{P}_{\alpha}:H)$ for some prime ideal $\mc{P}_{\alpha}$ of $A$. Thus $\mf{s} = \bigcap \mf{p}_{\alpha} = \bigcap (\mc{P}_{\alpha}:H) = \big( \bigcap \mc{P}_{\alpha}:H \big)$ with $\bigcap \mc{P}_{\alpha}$ a semiprime ideal of $A$. This completes the proof.
\end{proof}

\esp

\begin{cor} \label{radprimoeHradprimo}
Suppose that $H$ is finite-dimensional. For any $H$-stable ideal $I$ of $A$ we have $\Hrz{I} = (\rz{I}:H)$. In particular, $P_H(A) = (P(A):H)$ and $A$ is $H$-semiprime if and only if $P(A)$ does not contain nonzero $H$-stable ideals.
\end{cor}

\begin{proof}
By Proposition \ref{IdprimoIdHprimo}, $(\rz{I} : H)$ is an $H$-semiprime ideal. Moreover, $I$ is an $H$-stable ideal contained in $\rz{I}$, so $I \subseteq (\rz{I} : H)$. By Corollary \ref{HradmenorHsemiprimo}, $\Hrz{I} \subseteq (\rz{I} : H)$. On the other hand, again by Proposition \ref{IdprimoIdHprimo}, $\Hrz{I} = (\mc{Q} : H)$ for some semiprime ideal $\mc{Q}$ of $A$. In particular $I \subseteq \mc{Q}$, so $\rz{I} \subseteq \mc{Q}$ and therefore $(\rz{I}:H) \subseteq (\mc{Q}:H) = \Hrz{I}$.
\end{proof}
\esp

Now we establish a link between the $H$-prime (resp. $H$-semiprime) ideals of $A$ and the $\dual{H}$-prime (resp. $\dual{H}$-semiprime) ideals of $\psmash{A}{H}$, when $H$ is finite-dimensional. This allow us give a more precise description of the maps $\Phi$ and $\Psi$ defined in Corollary \ref{idHest2}. In particular, we establish a relation between the $H$-prime radical of $A$ and the $\dual{H}$-prime radical of $\psmash{A}{H}$.

\esp
\begin{teo} \label{HsemiprimoeHcosemiprimo}
Let $H$ be a finite-dimensional Hopf algebra and let $A$ be a partial $H$-module algebra. Then the restrictions of the maps $\Phi$ and $\Psi$ defined in \emph{Corollary \ref{idHest2}} are bijections between the set of the $H$-prime (resp. $H$-semiprime) ideals of $A$ and the set of the $\dual{H}$-prime (resp. $\dual{H}$-semiprime) ideals of $\psmash{A}{H}$. In particular,
$$P_H(A) = P_{\dual{H}} \!\! \paren{\psmash{A}{H}} \cap A \quad \quad \text{ and } \quad \quad P_{\dual{H}} \!\! \paren{\psmash{A}{H}} = \psmash{P_H(A)}{H} \, ,$$
so $A$ is an $H$-semiprime partial $H$-module algebra if and only if $\psmash{A}{H}$ is an $\dual{H}$-semiprime $\dual{H}$-module algebra.
\end{teo}

\begin{proof}
We will denote by $R = \psmash{A}{H}$. From Corollary \ref{idHest2}, it is sufficient to prove that, for any $H$-prime (resp. $H$-semiprime) ideal $\mf{p}$ of $A$, its image $\Phi(\mf{p})$ is an $\dual{H}$-prime (resp. $\dual{H}$-semiprime) ideal of $R$ and analogous for $\Psi$.

In fact, let $\mf{p}$ be an $H$-prime ideal of $A$ and let $\mc{I}, \mc{J}$ be $\dual{H}$-stable ideals of $R$ such that $\mc{I} \mc{J} \subseteq \Phi(\mf{p})$. Because $\Phi$ and $\Psi$ preserve inclusions and products, $\Psi(\mc{I})$ and $\Psi(\mc{J})$ are $H$-stable ideals of $A$ such that $\Psi(\mc{I}) \Psi(\mc{J}) = \Psi(\mc{I}\mc{J}) \subseteq \Psi(\Phi(\mf{p})) = \mf{p}$. Since $\mf{p}$ is $H$-prime, $\Psi(\mc{I}) \subseteq \mf{p}$ or $\Psi(\mc{J}) \subseteq \mf{p}$, which implies that $\mc{I} = \Phi(\Psi(\mc{I})) \subseteq \Phi(\mf{p})$ or $\mc{J} = \Phi(\Psi(\mc{J})) \subseteq \Phi(\mf{p})$. Hence $\Phi(\mf{p})$ is $\dual{H}$-prime. The proof that the $H$-semiprimality of an $H$-stable ideal $\mf{p}$ of $A$ implies the $\dual{H}$-semiprimality of $\Phi(\mf{p})$ is analogous and the argumentation for $\Psi$ is similar.

Now the equalities
$$P_H(A) = P_{\dual{H}} \! (R) \cap A \quad \quad \text{and} \quad \quad P_{\dual{H}} \! (R) = \psmash{P_H(A)}{H}$$
follow from $P_H(A)$ being the intersection of the all $H$-prime ideals of $A$, $P_{\dual{H}}(R)$ being the intersection of the all $\dual{H}$-prime ideals of $R$ and also because $\Phi, \Psi$  preserve intersections.
\end{proof}
\esp


\subsection{$H$-primitive ideals and the $H$-Jacobson radical} \label{HradJacobson}

Now we introduce the concept of $H$-primitive ideal in the context of partial actions. The definition of $H$-Jacobson radical of a partial $H$-module algebra  becomes natural from this concept (see Definition \ref{HradJac} below). The aim here is to establish  relations (if any) between the $H$-Jacobson radical of $A$ and the $\dual{H}$-Jacobson radical of $\psmash{A}{H}$, when $H$ is finite-dimensional. We start with the following definition.

\esp
\begin{defi}
Let $M$ be a partial $\modpar{A}{H}$-module.
\begin{itemize}
\item[\itm{(i)}] $M$ is called irreducible if $M \neq 0$ and $M$ has no  other partial $\modpar{A}{H}$-submodules than $0$ or $M$.
\item[\itm{(ii)}] $M$ is called  semisimple (or completely reducible) if $M$ is a direct sum of irreducible partial $\modpar{A}{H}$-submodules.
\end{itemize}
\end{defi}
\esp

The next result is immediate from Proposition \ref{AHmod1}.

\esp
\begin{prop} \label{AHmodsimples}
$M$ is an irreducible (resp. semisimple) partial $\modpar{A}{H}$-module if and only if $M$ is an irreducible (resp. semisimple) $\psmash{A}{H}$-module.
\end{prop}
\esp

The concept of irreducible partial $\modpar{A}{H}$-module naturally lead us to the concept of $H$-primitive partial $H$-module algebra and $H$-primitive ideal.

\esp
\begin{defi}
Let $A$ be a partial $H$-module algebra.
\begin{itemize}
\item[\itm{(i)}] $A$ is called right (resp. left) $H$-primitive if there exists a faithful irreducible right (resp. left) partial $\modpar{A}{H}$-module.
\item[\itm{(ii)}] An $H$-stable ideal $I$ of $A$ is called right (resp. left) $H$-primitive ideal if $A/I$ is a right (resp. left) $H$-primitive partial $H$-module algebra.
\end{itemize}
\end{defi}
\esp

Evidently an $H$-stable ideal $I$ of $A$ is right (resp. left) $H$-primitive if and only if $I$ is the annihilator of an irreducible right (resp. left) partial $\modpar{A}{H}$-module. From Propositions \ref{AHmod1} and \ref{AHmodsimples} we have the following.

\esp
\begin{prop} \label{HprimitiveDualHprimitiv}
An $H$-stable ideal $I$ of $A$ is right (resp. left) $H$-primitive if and only if $I = \mc{P} \cap A$ for some right (resp. left) primitive ideal $\mc{P}$ of $\psmash{A}{H}$.
\end{prop}

\vspace{0.1cm}

\begin{prop} \label{HJradical1}
Let $\{ I_{\alpha} \}$ be the family of the right $H$-primitive ideals of $A$ and let $\{ K_{\beta} \}$ be the family of the left $H$-primitive ideals of $A$. Then
$$\bigcap I_{\alpha} = J \! \paren{\psmash{A}{H}} \cap A = \bigcap K_{\beta}.$$
\end{prop}

\begin{proof}
By Proposition \ref{HprimitiveDualHprimitiv}, for each $\alpha$, $I_{\alpha} = \mc{P}_{\alpha} \cap A$ for some right primitive ideal $\mc{P}_{\alpha}$ of $\psmash{A}{H}$ and, moreover, $\{ \mc{P}_{\alpha} \}$ is the family of the right primitive ideals of $\psmash{A}{H}$. Thus
$$\bigcap I_{\alpha} = \bigcap (\mc{P}_{\alpha} \cap A) = \left( \bigcap \mc{P}_{\alpha} \right) \cap A = J \! \paren{\psmash{A}{H}} \cap A.$$
The proof of the second equality is similar.
\end{proof}
\esp

Thus, the following definitions are completely natural.

\esp
\begin{defi} \label{HradJac}
The $H$-Jacobson radical $J_H(A)$ of a partial $H$-module algebra $A$ is defined as the intersection of all the right (or left) $H$-primitive ideals of $A$.
\end{defi}

\vspace{0.1cm}

\begin{defi} \label{defHsemiprimitiv}
A partial $H$-module algebra $A$ is called $H$-semiprimitive if $J_H(A) = 0$.
\end{defi}
\esp


The next result in which we are interested says that if $\mpzc{P}$ is a right primitive ideal of $A$ then $(\mpzc{P}:H)$ is a right $H$-primitive ideal (Corollary \ref{IdprimitIdHprimit}). For this, we need to construct a right partial $\modpar{A}{H}$-module $W$ from a given right $A$-module $V$ such that $V$ can be seen as an $A$-submodule of $W$ which generates $W$ as a partial $\modpar{A}{H}$-module. We will describe this construction below.

Let $V$ be a right $A$-module. Consider the subspace $W$ of $V \ts_{\kk} H$ which is generated by the elements of the form
$$\sum v(k_1 \cdot x) \ts k_2, \qquad v \in V, \, x \in A, \, k \in H.$$
Then we define the right action of $A$ on $W$ as follows: for $v \in V$, $x,a \in A$ and $k \in H$,
\begin{eqnarray*}
\paren{\sum v(k_1 \cdot x) \ts k_2}a & := & \sum v(k_1 \cdot x)(k_2 \cdot a) \ts k_3 \\
& = & \sum v(k_1 \cdot (xa)) \ts k_2.
\end{eqnarray*}
It is clear that this action define an $A$-module structure on $W$. Moreover, we have $V \cong V \ts_{\kk} 1_H \subseteq W$ as $A$-module because, for any $v \in V$ and $a \in A$, we have
$$(v \ts 1_H)a = (v(1_H \cdot 1_A) \ts 1_H)a = v(1_H \cdot a) \ts 1_H = va \ts 1_H.$$

Now we define the right action of $H$ on $W$ in the following way: for $v \in V$, $x \in A$ and $k, h \in H$,
\begin{eqnarray*}
\paren{\sum v(k_1 \cdot x) \ts k_2} \acd h & := & \sum v(k_1 \cdot x)((k_2 h_1) \cdot 1_A) \ts k_3 h_2 \\
& = & \sum v(k_1 \cdot (x(h_1 \cdot 1_A)) \ts k_2 h_2.
\end{eqnarray*}
It is clear that $w \acd 1_H = w$ for any $w \in W$. Moreover, if $w = \sum v(k_1 \cdot x) \ts k_2$, with $v \in V$, $x \in A$ and $k \in H$, then, for any $a \in A$ and $h,g \in H$, we have
\begin{eqnarray*}
((w \acd h)a) \acd g & = & \paren{\paren{\sum v(k_1 \cdot x) \big( (k_2 h_1) \cdot 1_A \big) \ts k_3 h_2}a} \acd g \\
& = & \paren{\sum v(k_1 \cdot x) \big( (k_2 h_1) \cdot a \big) \ts k_3 h_2} \acd g \\
& = & \sum v(k_1 \cdot x) \big( (k_2 h_1) \cdot a \big) \big( (k_3 h_2 g_1) \cdot 1_A \big) \ts k_4 h_3 g_2 \\
& = & \sum v \Big( k_1 \cdot \big( x(h_1 \cdot a) \big) \Big) \big( (k_2 h_2 g_1) \cdot 1_A \big) \ts k_3 h_3 g_2 \\
& = & \paren{\sum v \Big( k_1 \cdot \big( x(h_1 \cdot a) \big) \Big) \ts k_2} \acd (h_2 g) \\
& = & \paren{\paren{\sum v(k_1 \cdot x) \ts k_2}(h_1 \cdot a)} \acd (h_2 g) \\
& = & (w (h_1 \cdot a)) \acd (h_2 g).
\end{eqnarray*}

Hence $W$ is a right partial $\modpar{A}{H}$-module which contain $V$ as $A$-submodule. Since $W$ is generated by elements of the form
$\sum v(k_1 \cdot x) \ts k_2$, we can see that
\begin{eqnarray*}
\sum v(k_1 \cdot x) \ts k_2 & = & \sum v(k_1 \cdot x)(k_2 \cdot 1_A) \ts k_3 \\
& = & \sum (v(k_1 \cdot x) \ts 1_H) \acd k_2 \\
& = & \sum \big( (v \ts 1_H)(k_1 \cdot x) \big) \acd k_2 \; \in \; (VA) \acd H,
\end{eqnarray*}
$v \in V$, $x \in A$ and $k \in H$, and it follows that $V$ generates $W$ as partial $\modpar{A}{H}$-module.

We are particularly interested in the case when $V$ is an irreducible right $A$-module.

\esp
\begin{prop} \label{AmodsimpAHmodsimp}
Let $V$ be an irreducible right $A$-module. Then there exists an irreducible right partial $\modpar{A}{H}$-module $M$, such that $V$ is (isomorphic to) an $A$-submodule of $M$. Moreover, if $H$ and $V$ are finite-dimensional then $M$ is also finite-dimensional and
$$\dim_{\kk}(M) \leq \dim_{\kk}(H) \dim_{\kk}(V).$$
\end{prop}

\begin{proof}
Let $W$ be a right partial $\modpar{A}{H}$-module such that $V$ is an $A$-submodule of $W$ and $V$ generates $W$ as partial $\modpar{A}{H}$-module. Then $W = (VA) \acd H = V \acd H$. Since $V$ is an irreducible $A$-module, for any $0 \neq u \in V$ we have $V = uA$, and so $W = (uA) \acd H$. Thus every nonzero element of $V$ generates $W$ as partial $\modpar{A}{H}$-module.

Now, by Zorn's Lemma, there exists a partial $\modpar{A}{H}$-submodule $U$ of $W$ maximal with respect to the property that $U \cap V = 0$. Then the quotient $M = W/U$ has a natural right partial $\modpar{A}{H}$-module structure. Moreover, from $U \cap V = 0$, we have a natural monomorphism of right $A$-modules:
\begin{eqnarray*}
\ffuni{V}{M = W/U}{v}{\br{v} = v + U}.
\end{eqnarray*}

We claim that $M$ is irreducible (as partial $\modpar{A}{H}$-module). In fact, if $N$ is a nonzero partial $\modpar{A}{H}$-submodule of $M$, then $N = T/U$ for some partial $\modpar{A}{H}$-submodule $T$ of $W$ such that $U \varsubsetneq T$. By maximality of $U$, we have $T \cap V \neq 0$. Fix a nonzero element $u \in T \cap V$. As observed above, $W = (uA) \acd H$, so $M = (\br{u}A) \acd H \subseteq N$ and $N = M$. Hence $M$ is an irreducible right partial $\modpar{A}{H}$-module and $V$ is isomorphic to an $A$-submodule of $M$.

Clearly $M = (VA) \acd H = V \acd H$ and the last statement follows.
\end{proof}

\esp

\begin{cor} \label{IdprimitIdHprimit}
If $\mpzc{P}$ is a right primitive ideal of $A$, then $(\mpzc{P}:H)$ is right $H$-primitive.
\end{cor}

\begin{proof}
Let $V$ be an irreducible right $A$-module such that $\mpzc{P} = \ann(V)$ and let $M$ be an irreducible right partial $\modpar{A}{H}$-module such that $V$ is an $A$-submodule of $M$, as in the Proposition \ref{AmodsimpAHmodsimp}. Denote $I := (\mpzc{P}:H)$. For any $v \in V$, $x \in I$ and $h \in H$, we have
$$(v \acd h)x = \sum (v(h_1 \cdot x)) \acd h_2 \; \in \; (V(H \cdot I)) \acd H \subseteq (V \mpzc{P}) \acd H = 0,$$
therefore $MI = (V \acd H)I = 0$ and $(\mpzc{P}:H) = I \subseteq \ann(M)$. On the other hand, $\ann(M)$ is an $H$-stable ideal of $A$ (Proposition \ref{anulHideal}) which is contained in $\mpzc{P}$ (since $V \subseteq M$), so $\ann(M) \subseteq (\mpzc{P}:H)$. Hence $(\mpzc{P}:H) = \ann(M)$ is a right $H$-primitive ideal of $A$.
\end{proof}
\esp


From now on we restrict ourselves to the case when $H$ is finite-dimensional. In this situation, we can obtain more details about the $H$-Jacobson radical of $A$. First, however, as made before for right partial $\modpar{A}{H}$-module, we will present a construction of a left partial $\modpar{A}{H}$-module $W$ from a given left $A$-module $V$, such that $V$ can be seen as an $A$-submodule of $W$ which generates $W$ as a partial $\modpar{A}{H}$-module. This construction is an adaptation for partial actions of one which appeared in \cite{LM} and we will need a nonzero integral element in $\dual{H}$ so that
$H$ need to be finite-dimensional.


Suppose that $H$ is a finite-dimensional Hopf algebra and let $A$ be a partial $H$-module algebra. In this case,  analogously to the global case, the (left) partial $H$-module algebra structure of $A$ induces a structure of (right) partial $\dual{H}$-comodule algebra $\rho : A \rightarrow A \otimes H^*$ such that
\begin{eqnarray} \label{actcoact}
\rho(a) = \sum a_0 \otimes a_1 \qquad \Longleftrightarrow \qquad h \cdot a = \sum a_0 a_1(h) \, , \,\,\, \forall \, h \in H.
\end{eqnarray}
(see \cite{CJ} or \cite{AB3}). Remember also that $\dual{H}$ is an $H$-module algebra (global action) via $h \rightharpoonup \varphi = \sum \varphi_1 \varphi_2(h)$, for $h \in H$ and $\varphi \in \dual{H}$. Then, for $a \in A$ and $h \in H$, we have
\begin{eqnarray*}
\rho(h \cdot a) & = & \sum \rho(a_0) a_1(h) \qquad \qquad \text{(by (1))} \\
& = & \sum a_{00} \otimes a_{01} a_1(h) \\
& = & \sum 1_0 a_0 \otimes 1_1 a_{11} a_{12}(h) \qquad \qquad \text{(by \itm{(PC3)}, Definition \ref{defHcomodAlgPar})} \\
& = & \sum 1_0 a_0 \otimes 1_1 (h \rightharpoonup a_1) \\
& = & \rho(1_A) (id_A \ts (h \rightharpoonup)) \rho(a)
\end{eqnarray*}

Now let $V$ be a left $A$-module and consider the subspace $W = \rho(A) (V \ts \dual{H})$ of $V \ts \dual{H}$ generated by elements of the form
$$\rho(x)(v \ts \varphi) = \sum x_0 v \ts x_1 \varphi \, , \qquad x \in A, \, v \in V, \, \varphi \in \dual{H}.$$
We define the action of $A$ and $H$ on $W$ as follows: for $w \in W$, $a \in A$ and $h \in H$,
$$a \bullet w = \rho(a) w \qquad \text{and} \qquad h \ace w = \rho(1_A) (id_V \ts (h \rightharpoonup))(w),$$
that is, if $w = \sum x_0 v \ts x_1 \varphi$ with $x \in A$, $v \in V$ and $\varphi \in \dual{H}$, then
$$a \bullet w := \sum a_0 x_0 v \ts a_1 x_1 \varphi = \rho(ax)(v \ts \varphi) \qquad \qquad \text{(by \itm{(PC2)}, Definition \ref{defHcomodAlgPar})}$$
and
$$h \ace w = \rho(1_A) \left( \sum x_0 v \ts h \rightharpoonup (x_1 \varphi) \right) = \sum 1_0 x_0 v \ts 1_1 (h \rightharpoonup (x_1 \varphi)).$$

Evidently $\bullet$ defines a left $A$-module structure on $M$. Also it is clear that \itm{(PM1)} of the Definition \ref{defAHmod} is satisfied. For \itm{(PM2)}, let $w = \sum x_0 v \ts x_1 \varphi \in W$, with $x \in A$, $v \in V$, $\varphi \in \dual{H}$, and let $a \in A$, $h,g \in H$. Then we have
\begin{eqnarray*}
h \ace(a \bullet (g \ace w)) & = & h \ace \paren{a \bullet \paren{\sum 1_0 x_0 v \ts 1_1 (g \rightharpoonup (x_1 \varphi))}} \\
& = & h \ace \paren{\sum a_0 x_0 v \ts a_1 (g \rightharpoonup (x_1 \varphi))} \\
& = & \sum 1_0 a_0 x_0 v \ts 1_1 h \rightharpoonup (a_1 (g \rightharpoonup (x_1 \varphi))) \\
& = & \sum 1_0 a_0 x_0 v \otimes 1_1 (h_1 \rightharpoonup a_1) ((h_2 g) \rightharpoonup (x_1 \varphi)) \\
& = & \sum \rho(h_1 \cdot a) [x_0 v \otimes (h_2 g) \rightharpoonup (x_1 \varphi)] \\
& = & \sum (h_1 \cdot a) \bullet ((h_2 g) \ace w).
\end{eqnarray*}
Hence $W$ is a left partial $\modpar{A}{H}$-module algebra.

Now, let $\lambda$ be a nonzero left integral in $\dual{H}$. Then we have an isomorphism of left $A$-modules $V \cong V \ts \lambda$. In fact, by \itm{(PC1)} of Definition \ref{defHcomodAlgPar}, for every $v \in V$ and $a \in A$,
$$v \ts \lambda = 1 v \ts \lambda = 1_0 v \ts \ep{\dual{H}}(1_1) \lambda = \sum 1_0 v \ts 1_1 \lambda = \rho(1_A)(v \ts \lambda) \, \in \, W$$
and, by the same reason,
$$a \bullet (v \ts \lambda) = av \ts \lambda.$$
Moreover, by a result of Larson and Sweedler \cite{LS}, the map $H \to \dual{H}$, given by $h \mapsto (h \rightharpoonup \lambda)$, is an isomorphism. Consequently, for every $x \in A$, $v \in V$ and $\varphi \in \dual{H}$, there exists $h \in H$ such that $h \rightharpoonup \lambda = \varphi$, and then
$$\sum x_0 v \ts x_1 \varphi = \rho(x)(v \ts (h \rightharpoonup \lambda)) = x \bullet (h \ace (v \ts \lambda)) \, \in \, A (H \ace V),$$
via $V \cong V \ts \lambda$, so $V$ generates $W$ as left partial $\modpar{A}{H}$-module.

\esp
\begin{remark} \label{leftAHmodext}
Observe that if $V$ and $W$ are as above and $V$ is irreducible, then $W = A(H \ace (v \ts \lambda))$ for any nonzero $v \in V$ and we can obtain the analogous of the \emph{Proposition \ref{AmodsimpAHmodsimp}} for irreducible left $A$-modules. The proof is essentially the same (with the appropriate adaptations), except for the inequality $\dim_{\kk}(M) \leq \dim_{\kk}(H) \dim_{\kk}(V)$. This latter follows since $H$ is finite-dimensional and $M = A(H \ace (v \ts \lambda))$, so there exists an epimorphism of $A$-modules $V^n \to M$. With regard to the \emph{Corollary \ref{IdprimitIdHprimit}} we can only deduce that $\ann(M) \subseteq (\mpzc{P}:H)$. However, this inclusion is enough for our purposes (see \emph{Theorem \ref{Semiprimitiv1}}).
\end{remark}
\esp


In the sequel we give an important characterization for the $H$-Jacobson radical when $H$ is finite-dimensional.

\esp
\begin{prop} \label{HJradical2}
Suppose that $H$ is finite-dimensional. Then $J_H(A) = (J(A):H)$. In particular, $A$ is $H$-semiprimitive if and only if $J(A)$ does not contain nonzero $H$-stable ideals.
\end{prop}

\begin{proof}
By Proposition \ref{HJradical1} and Remark \ref{JacA}, we have that $J_H(A) = J(\psmash{A}{H}) \cap A \subseteq J(A)$. Since $J_H(A)$ is $H$-stable then the inclusion
$J_H(A) \subseteq (J(A):H)$ follows. Also, we observe that it is possible to  deduce this same inclusion from Corollary \ref{IdprimitIdHprimit}.

To prove the reverse inclusion, we will show that $\mpzc{J} := (J(A):H) \subseteq \ann(M)$, for any irreducible left partial $\modpar{A}{H}$-module $M$. In fact, if $0 \neq m \in M$, then $A(H \ace m)$ is a nonzero partial $\modpar{A}{H}$-submodule of $M$, so $A(H \ace m) = M$. Since $H$ is finite-dimensional, this implies that $M$ is finitely generated as $A$-module. From $\mpzc{J} \subseteq J(A)$, by Nakayama's Lemma, $\mpzc{J} M \varsubsetneq M$. Moreover, $\mpzc{J}$ is an $H$-stable ideal, so $\mpzc{J} M$ is a partial $\modpar{A}{H}$-submodule of $M$. Hence $\mpzc{J} M = 0$ and $(J(A):H) = \mpzc{J} \subseteq \ann(M)$.
\end{proof}
\esp

We observe that if $H$ is finite-dimensional and the action of $H$ on $A$ is global, then we can recover \cite[Corollary 2.6 (2)]{LMS} by using
the Proposition \ref{HJradical2} combined with the Definitions \ref{HradJac} and \ref{defHsemiprimitiv}.

From Corollary \ref{IdprimitIdHprimit} (resp. Remark \ref{leftAHmodext}) and Proposition \ref{HJradical2}, when $H$ is finite-dimensional, it is sufficient to take the intersection of all $H$-stable ideals of $A$ which are of the form $(\mpzc{P}:H)$, where $\mpzc{P}$ is a right (resp. left) primitive ideal of $A$, to obtain $J_H(A)$ as it will be shown in the next result.

\esp
\begin{cor} \label{JHAidprimit}
Let $\{ \mpzc{P}_{\alpha} \}$ be the family of all right (resp. left) primitive ideals of $A$. If $H$ is finite-dimensional, then $J_H(A) = \bigcap (\mpzc{P}_{\alpha}:H)$.
\end{cor}

\begin{proof}
By Corollary \ref{IdprimitIdHprimit} (resp. Remark \ref{leftAHmodext}), it is clear that $J_H(A) \subseteq \bigcap (\mpzc{P}_{\alpha}:H)$. On the other hand, $\bigcap (\mpzc{P}_{\alpha}:H) \subseteq \bigcap \mpzc{P}_{\alpha} = J(A)$. Since $\bigcap (\mpzc{P}_{\alpha}:H)$ is $H$-stable, by Proposition \ref{HJradical2}, $\bigcap (\mpzc{P}_{\alpha}:H) \subseteq (J(A):H) = J_H(A)$.
\end{proof}
\esp


We finish this section by establishing a relation between the $H$-Jacobson radical of $A$ and the $\dual{H}$-Jacobson radical of $\psmash{A}{H}$, when $H$ is finite-dimensional.

\esp
\begin{teo} \label{Hsemiprimitiv}
Let $H$ be a finite-dimensional Hopf algebra and let $A$ be a partial $H$-module algebra. Then
$$J_H(A) = J_{\dual{H}} \!\! \paren{\psmash{A}{H}} \cap A \quad \quad \text{and} \quad \quad J_{\dual{H}} \!\! \paren{\psmash{A}{H}} = \psmash{J_H(A)}{H}.$$
In particular, $A$ is $H$-semiprimitive if and only if $\psmash{A}{H}$ is $\dual{H}$-semiprimitive.
\end{teo}

\begin{proof}
We will denote by $R := \psmash{A}{H}$. Let $\Phi$ and $\Psi$  as in the Corollary \ref{idHest2}. Since $J_{\dual{H}}(R)$ is $\dual{H}$-stable and $J_{\dual{H}}(R) \subseteq J(R)$,  by Proposition \ref{HJradical1}, we have
\begin{eqnarray*}
J_{\dual{H}}(R) & = & \Phi(\Psi(J_{\dual{H}}(R))) \,\, = \,\, \psmash{(J_{\dual{H}}(R) \cap A)}{H} \\
& \subseteq & \psmash{(J(R) \cap A)}{H} \,\, = \,\, \psmash{J_H(A)}{H}.
\end{eqnarray*}
On the other hand, also by Proposition \ref{HJradical1},
$$\psmash{J_H(A)}{H} = J_H(A)R \subseteq J(R)R = J(R).$$
Since $\psmash{J_H(A)}{H}$ is an $\dual{H}$-stable ideal, by Proposition \ref{HJradical2}, it follows that $$\psmash{J_H(A)}{H} \subseteq (J(R):\dual{H}) = J_{\dual{H}}(R).$$
Thus the second equality holds. The first one follows from the properties of $\Phi$ and $\Psi$ because
$$J_H(A) = \Psi(\Phi(J_H(A))) = \psmash{J_H(A)}{H} \cap A = J_{\dual{H}}(R) \cap A.$$
\end{proof}
\esp


\section{On the semiprimitivity and the semiprimality problems for the partial smash products} \label{semiprimitivProdSmash}

In this section we apply the results of the Sections \ref{IdHesteAHmodpar} and \ref{Hradicais} to investigate the semiprimitivity and the semiprimality problems mentioned in the introduction, for  partial smash products. Most of the results here generalize to the case of partial actions, the corresponding results on the semiprimitivity of the (global) smash product which appear in \cite[Section 4]{LMS}, improving them when $\kk$ has positive characteristic. We also prove a result about the semiprimality of the partial smash product (Theorem \ref{semiprimal1}), which generalizes \cite[Theorem 3.4]{LM} for the case of partial actions.


\esp
\begin{teo} \label{Semiprimitiv1} \emph{(See \cite[Theorem 4.1]{LMS})}
Let $H$ be a semisimple Hopf algebra and let $A$ be an $H$-semiprimitive partial $H$-module algebra. If every irreducible right (or left) $A$-module is finite-dimensional, then $\psmash{A}{H}$ is semiprimitive.
\end{teo}

\begin{proof} We will denote by $R := \psmash{A}{H}$. Let $\{ V_{\alpha} \}$ be the family of all irreducible (right) $A$-modules. By Proposition \ref{AmodsimpAHmodsimp}, for each $\alpha$, there exists an irreducible (right) partial $\modpar{A}{H}$-module $M_{\alpha}$ such that $V_{\alpha}$ is an $A$-submodule of $M_{\alpha}$. For each $\alpha$, we will denote by $\mpzc{P}_{\alpha} := \ann(V_{\alpha})$, $I_{\alpha} := \ann(M_{\alpha}) \subseteq (\mpzc{P}_{\alpha}:H)$ (see the proof of the Corollary \ref{IdprimitIdHprimit} and Remark \ref{leftAHmodext}) and $A_{\alpha} := A/I_{\alpha}$. By Corollary \ref{JHAidprimit}, we have
$$\bigcap I_{\alpha} \subseteq \bigcap (\mpzc{P}_{\alpha}:H) = J_H(A) = 0.$$
Thus $A$ is a subdirect product of the partial $H$-module algebras $A_{\alpha}$. By Proposition \ref{ProdSubdirSmash}, $R$ is a subdirect product of the algebras $R_{\alpha} := \psmash{A_{\alpha}}{H}$. Therefore, it is enough to show that each $R_{\alpha}$ is semiprimitive (see \cite[Proposition 2.3.4]{GW}).

Since $V_{\alpha}$ is finite-dimensional, so is $M_{\alpha}$ (Proposition \ref{AmodsimpAHmodsimp} or Remark \ref{leftAHmodext}). Moreover, from $I_{\alpha} = \ann(M_{\alpha})$ it follows that $A_{\alpha} = A/I_{\alpha} \subseteq \End_{\kk}(M_{\alpha})$ as algebras. Therefore $\dim_{\kk}(A_{\alpha}) \leq (\dim_{\kk}(M_{\alpha}))^2 < \infty$ and so $R_{\alpha}$ is finite-dimensional.

Since $A_{\alpha}$ is $H$-primitive it follows  by Theorem \ref{Hsemiprimitiv} that $R_{\alpha}$ is $\dual{H}$-semiprimitive. Moreover, because $\dual{H}$ is cosemisimple Hopf algebra acting on $R_\alpha$ globally and $R_{\alpha}$ is finite-dimensional, we deduce that $J(R_{\alpha})$ is an $\dual{H}$-stable ideal, by \cite[Corollary 3.2]{LM}. Hence
$$J(R_{\alpha}) = (J(R_{\alpha}):H) = J_{\dual{H}}(R_{\alpha}) = 0$$
and $R_{\alpha}$ is semiprimitive as desired.
\end{proof}
\esp


Now we observe that if $H$ is a Hopf algebra over a field $\kk$, $A$ is a partial $H$-module algebra and $\kk \subseteq \FF$ is any field extension,
then $H^{\FF} := H \ts_{\kk} \FF$ becomes a Hopf algebra over $\FF$ with coalgebra structure given by
$$\Delta_{H^{\FF}}(h \ts_{\kk} \alpha) = \sum (h_1 \ts_{\kk} \alpha) \ts_{\FF} (h_2 \ts_{\kk} 1_{\FF}) = \sum (h_1 \ts_{\kk} 1_{\FF}) \ts_{\FF} (h_2 \ts_{\kk} \alpha)$$
and
$$\varepsilon_{H^{\FF}}(h \ts_{\kk} \alpha) = \varepsilon_H(h) \alpha,$$
for $h \in H$ and $\alpha \in \FF$, where $\varepsilon_H$ is the counit of $H$. The antipode of $H^{\FF}$ is given by
$$S_{H^{\FF}}(h \ts_{\kk} \alpha) = S_H(h) \ts_{\kk} \alpha, \qquad h \in H, \,\, \alpha \in \FF,$$
where $S_H$ is the antipode of $H$. Moreover, $A^{\FF} := A \ts_{\kk} \FF$ becomes a partial $H^{\FF}$-module algebra by the
$H^{\FF}$-action defined as
$$(h \tssp \alpha) \cdot (a \tssp \beta) := (h \cdot a) \tssp (\alpha \beta), \qquad \alpha, \beta \in \FF, \,\, a \in A, \,\, h \in H.$$

With the above notations we can prove the following result.

\esp
\begin{lema} \label{extsepar}
Let $H$ be a finite-dimensional Hopf algebra over a field $\kk$ and let $A$ be a partial $H$-module algebra. If $\kk \subseteq \FF$ is a separable algebraic field extension, then
$$J_{H^{\FF}}(A^{\FF}) = J_H(A) \ts_{\kk} \FF.$$
In particular, $A$ is $H$-semiprimitive if and only if $A^{\FF}$ is $H^{\FF}$-semiprimitive.
\end{lema}

\begin{proof}
Since $\kk \subseteq \FF$ is a separable algebraic extension we have $J(A^{\FF}) = J(A) \ts_{\kk} \FF$ (see \cite[Theorem 5.17]{L1}). Thus
$$H^{\FF} \cdot (J_H(A) \ts_{\kk} \FF) = (H \cdot J_H(A)) \ts_{\kk} \FF \subseteq J(A) \ts_{\kk} \FF = J(A^{\FF}).$$
By Proposition \ref{HJradical2}, $J_H(A) \ts_{\kk} \FF \subseteq (J(A^{\FF}):H^{\FF}) = J_{H^{\FF}}(A^{\FF})$.

On the other hand, if $u \in J_{H^{\FF}}(A^{\FF}) \subseteq J(A^{\FF}) = J(A) \ts_{\kk} \FF$, then there are $x_1, \ldots, x_n \in J(A)$ and $\beta_1, \ldots, \beta_n \in \FF$, with $\{ \beta_i \}$ linearly independent over $\kk$, such that $u = \sum x_i \ts \beta_i$. So, for any $h \in H$,
$$\sum (h \cdot x_i) \ts \beta_i = (h \ts 1_{\FF}) \cdot u \,\,\, \in \, J_{H^{\FF}}(A^{\FF}) \, \subseteq \, J(A) \ts_{\kk} \FF$$
(since $J_{H^{\FF}}(A^{\FF})$ is $H^{\FF}$-stable). From the linear independence of $\{ \beta_i \}$, we can deduce that $h \cdot x_i \in J(A)$, for every $1 \leq i \leq n$ and every $h \in H$. Therefore $x_i \in (J(A):H) = J_H(A)$ for every $i$, and $u \in J_H(A) \ts_{\kk} \FF$. Hence, also the inclusion $J_{H^{\FF}}(A^{\FF}) \subseteq J_H(A) \ts_{\kk} \FF$ holds.
\end{proof}

Now we are able to present our second result about the semiprimitivity of the partial smash product.

\esp

\begin{teo} \label{Semiprimitiv2} \emph{(See \cite[Theorem 4.2]{LMS})}
Let $H$ be a semisimple Hopf algebra over a field $\kk$ and let $A$ be an $H$-semiprime partial $H$-module algebra satisfying a polynomial identity. If $\kk$ is perfect and $A$ is affine over $\kk$, then $R = \psmash{A}{H}$ is semiprimitive.
\end{teo}

\begin{proof}
Since $A$ is a PI-algebra which is affine over $\kk$, we have that $J(A) = P(A)$ (see \cite[Corollary 4.4.6]{R}). Thus, $A$ is actually $H$-semiprimitive. Denote by $\bar{\kk}$ the algebraic closure of $\kk$. As $\kk$ is a perfect field it follows that the extension $\kk \subseteq \bar{\kk}$ is separable. Taking $\FF = \bar{\kk}$ in the Lemma \ref{extsepar}, we have that $\bar{A}$ is $\bar{H}$-semiprimitive, where $\bar{H} = H \ts_{\kk} \bar{\kk}$ and $\bar{A} = A \ts_{\kk} \bar{\kk}$.

Moreover, $\bar{A}$ is a PI-algebra (see \cite[Theorem 6.1.1]{R}) which is affine over $\bar{\kk}$, and so every irreducible right $\bar{A}$-module is finite-dimensional (see \cite[Lemma 3.7]{LMS}). Since $\bar{H}$ is semisimple (see \cite[Corollary 2.2.2]{M}), it follows from Theorem \ref{Semiprimitiv1} that $R \ts_{\kk} \bar{\kk} \cong \psmash{\bar{A}}{\bar{H}}$ is semiprimitive. Thus, by \cite[Theorem 5.17]{L1}, $R$ is semiprimitive.
\end{proof}
\esp


For the next result we will need one more lemma. Before, observe that if $I$ is a right ideal of $A$, then so is $H \cdot I$. This follows because, for any $h \in H$, $x \in I$ and $a \in A$,
$$(h \cdot x)a = \sum (h_1 \cdot x)((h_2 S(h_3)) \cdot a) = \sum h_1 \cdot (x (S(h_2) \cdot a)) \,\, \in \, H \cdot (IA) \, \subseteq \, H \cdot I,$$
so $(H \cdot I)A \, \subseteq H \cdot I$.

\esp
\begin{lema} \label{JacHest2} \emph{(See \cite[Remark 3.9]{LMS})}
Let $H$ be a finite-dimensional cosemisimple Hopf algebra and let $A$ be an $H$-module algebra (global action). If $A$ is locally finite, then $J(A)$ is $H$-stable.
\end{lema}

\begin{proof}
Fix an element $y \in H \cdot J(A)$, and let $x_1, \ldots, x_n \in J(A)$ and $h_1, \ldots, h_n \in H$  such that $y = \sum h_i \cdot x_i$. Since $\sum H \cdot x_i$ is a finite-dimensional subspace of $A$, it generates a finite-dimensional subalgebra $B$ of $A$. Clearly $H \cdot B \subseteq B$, so $B$ is an $H$-module subalgebra of $A$.

Since $A$ is locally finite, $J(A)$ is a nil ideal of $A$, therefore each $B x_i B \subseteq J(A)$ is a nil ideal of $B$ and so it is contained in $J(B)$. In particular, each $x_i \in J(B)$. Since $B$ is finite-dimensional, it follows from \cite[Corollary 3.2]{LM} that $H \cdot J(B) \subseteq J(B)$. Thus, we have that $y = \sum h_i \cdot x_i \in J(B)$ is a nilpotent element.

From this and the previous observation, it follows that $H \cdot J(A)$ is a nil right ideal of $A$, so it is contained in $J(A)$. Hence $J(A)$ is $H$-stable.
\end{proof}

\esp

\begin{teo} \label{Semiprimitiv3} \emph{(See \cite[Corollary 4.4]{LMS})}
Let $H$ be a semisimple Hopf algebra and let $A$ be an $H$-semiprimitive partial $H$-module algebra. If $A$ is locally finite, then $\psmash{A}{H}$ is semiprimitive.
\end{teo}

\begin{proof}
By Theorem \ref{Hsemiprimitiv}, the smash product $\psmash{A}{H}$ is an $\dual{H}$-semiprimitive $\dual{H}$-module algebra (global action). Also, $\dual{H}$ is a finite-dimensional cosemisimple Hopf algebra, because $H$ is semisimple. Moreover, $\psmash{A}{H}$ is locally finite, because if $x_1, \ldots, x_n \in \psmash{A}{H}$ and  $x_i = \sum_j \psmash{a_{ij}}{h_{ij}}$, $1\leq i \leq n$, then there is a finite-dimensional $H$-stable subalgebra $B$ of $A$ generated by $\sum_{i,j} (H \cdot 1_A)(H \cdot a_{ij})$, so that $x_1, \ldots, x_n$ are elements of the finite-dimensional subalgebra $\psmash{B}{H}$ of $\psmash{A}{H}$. By Lemma \ref{JacHest2}, $J(\psmash{A}{H})$ is $\dual{H}$-stable. Hence, by Proposition \ref{HJradical2},
$$J \! \paren{\psmash{A}{H}} = \paren{J \! \paren{\psmash{A}{H}}:\dual{H}} = J_{\dual{H}} \!\! \paren{\psmash{A}{H}} = 0$$
and $\psmash{A}{H}$ is semiprimitive.
\end{proof}

\esp

\begin{cor} \label{Semiprimitiv31} \emph{(See \cite[Theorem 4.2]{LMS})}
Let $H$ be a semisimple Hopf algebra over a field $\kk$ and let $A$ be an $H$-semiprime partial $H$-module algebra satisfying a polynomial identity. If $A$ is algebraic over $\kk$, then $\psmash{A}{H}$ is semiprimitive.
\end{cor}

\begin{proof}
Under these conditions, we have $J(A) = P(A)$ (see \cite[Corollary 4.19]{L1} and \cite[Theorem 3, p. 36]{J}), thus $A$ is actually $H$-semiprimitive. Moreover, $A$ is locally finite (see \cite[Theorem 6.4.3]{H}), therefore the result follows from Theorem \ref{Semiprimitiv3}.
\end{proof}

\esp

\begin{cor} \label{Semiprimitiv4}
Let $H$ be a semisimple Hopf algebra over a field $\kk$ and let $A$ be a partial $H$-module algebra. If
\begin{itemize}
\item[\itm{(1)}] every irreducible right $A$-module is finite-dimensional, or
\item[\itm{(2)}] $\kk$ is perfect and $A$ is a PI-algebra which is affine over $\kk$, or
\item[\itm{(3)}] $A$ is locally finite (in particular, if $A$ is a PI-algebra which is algebraic over $\kk$),
\end{itemize}
then $J \! \paren{\psmash{A}{H}} = \psmash{J_H(A)}{H}$.
\end{cor}

\begin{proof}
By Theorem \ref{Hsemiprimitiv}, $\psmash{J_H(A)}{H} = J_{\dual{H}} \!\! \paren{\psmash{A}{H}} \subseteq J \! \paren{\psmash{A}{H}}$. On the other hand, $A/J_H(A)$ is $H$-semiprimitive (so, in particular, $H$-semiprime). By Theorems \ref{Semiprimitiv1}, \ref{Semiprimitiv2} or \ref{Semiprimitiv3} (according the hypotheses), it follows that the factor algebra $\paren{\psmash{A}{H}}/ \big( \psmash{J_H(A)}{H} \big) \cong \psmash{(A/J_H(A))}{H}$ is semiprimitive, so that $J \! \paren{\psmash{A}{H}} \subseteq \psmash{J_H(A)}{H}$.
\end{proof}
\esp


The next theorem generalizes \cite[Theorem 3.4]{LM} for the case of partial actions. The proof follows almost the same steps as in \cite{LM}, but we include here our proof for the sake of completeness.

\esp
\begin{teo} \label{semiprimal1}
Let $H$ be a semisimple Hopf algebra and let $A$ be an $H$-semiprime partial $H$-module algebra satisfying a polynomial identity. Then $\psmash{A}{H}$ is semiprime.
\end{teo}

\begin{proof}
Since $A$ is a PI-algebra, if $I$ is a nil ideal of $A$ then $I \subseteq P(A)$, so $(I:H) \subseteq (P(A):H) = P_H(A) = 0$.

The polynomial ring $A[t]$ has a structure of partial $H$-module algebra if we extend the action of $H$ by $h \cdot t := (h \cdot 1_A)t$, for $h \in H$. Moreover, for any ideal $I$ of $A$, $(I[t]:H) = (I:H)[t]$, since
$$p(t) = \sum a_i t^i \; \in \, (I[t]:H) \quad \Longleftrightarrow \quad H \cdot p(t) \subseteq I[t] \quad \Longleftrightarrow \quad H \cdot a_i \subseteq I, \; \forall \, i$$
$$\Longleftrightarrow \quad a_i \in (I:H), \; \forall \, i \quad \Longleftrightarrow \quad p(t) \in (I:H)[t].$$
By a theorem of Amitsur (see \cite[Theorem 5.10]{L1}), $J(A[t]) = N[t]$ for some nil ideal $N$ of $A$, so
$$J_H(A[t]) = (J(A[t]):H) = (N[t]:H) = (N:H)[t] = 0$$
as observed above. Therefore $A[t] \cong A \ts_{\kk} \kk[t]$ is an $H$-semiprimitive partial $H$-module algebra which satisfies a polynomial identity (see \cite[Theorem 6.1.1]{R}).

It follows from Theorem \ref{Hsemiprimitiv} that $R := \psmash{A[t]}{H}$ is $\dual{H}$-semiprimitive. Moreover, $R$ satisfies a polynomial identity, because $R$ is finitely generated as $A[t]$-module (see \cite[Corollary 13.4.9]{MR}). Since $\dual{H}$ is cosemisimple, it follows from \cite[Corollary 3.2]{LM} that if $\mc{I}$ is a nilpotent ideal of $R$, then $\dual{H} \cdot \mc{I} \subseteq J(R)$, so $\mc{I} \subseteq (J(R):\dual{H}) = J_{\dual{H}}(R) = 0$. Thus $R = \psmash{A[t]}{H} \cong \psmash{A}{H}[t]$ is semiprime and therefore $\psmash{A}{H}$ is also semiprime (see \cite[Proposition 10.18]{L1}).
\end{proof}

\esp

\begin{cor} \label{semiprimal2}
Let $H$ be a semisimple Hopf algebra and let $A$ be a partial $H$-module algebra satisfying a polynomial identity. Then
$$P \! \paren{\psmash{A}{H}} = \psmash{P_H(A)}{H}.$$
\end{cor}

\begin{proof}
By Theorem \ref{HsemiprimoeHcosemiprimo}, $\psmash{P_H(A)}{H} = P_{\dual{H}} \!\! \paren{\psmash{A}{H}} \subseteq P \! \paren{\psmash{A}{H}}$. On the other hand, $A/P_H(A)$ is $H$-semiprime (and satisfies a polynomial identity). It follows from Theorem \ref{semiprimal1} that the factor algebra $\paren{\psmash{A}{H}}/ \big( \psmash{P_H(A)}{H} \big) \cong \psmash{(A/P_H(A))}{H}$ is semiprime, so also $P \! \paren{\psmash{A}{H}} \subseteq \psmash{P_H(A)}{H}$.
\end{proof}
\esp


The hypothesis ``$H$-semiprimitive'' or ``$H$-semiprime'' in the above theorems are essentials (by Theorems \ref{HsemiprimoeHcosemiprimo} and \ref{Hsemiprimitiv}, Corollary \ref{radprimoeHradprimo} and Proposition \ref{HJradical2}). Also, it is clear that the semisimplicity of $H$ is a
necessary condition for the Question \ref{qsemiprimit} in the introduction, because if $H$ acts trivially on $A$, then $A\# H = A\otimes H$ is not semiprime if $H$ is not semisimple.


\esp

\begin{ex}
Let $B$, $A$ and $H$ be as in \emph{Example \ref{AcParc1}}. If $B$ is finite-dimensional and semiprimitive, then so is $A$. In particular, every irreducible $A$-module is finite-dimensional. By \emph{Theorem \ref{Semiprimitiv1}}, the partial smash product $\psmash{A}{H}$ is semiprimitive. This also follows from \emph{Theorem \ref{semiprimal1}} since, for finite-dimensional algebra, the Jacobson radical coincides with the prime radical.
\end{ex}

\esp

{\bf Acknowledgments.}  The authors would like to thank  Antonio Paques who gave us the definition of left partial $(A,H)$-module.


\bibliographystyle{plain}

\bibliography{Bibliografia}

\end{document}